\documentclass[final]{dmtcs-episciences}

\usepackage[english]{babel} 
\usepackage[utf8]{inputenc}
\usepackage[round,comma]{natbib}

\author{Cyril Banderier\affiliationmark{1}\thanks{\url{http://lipn.fr/~banderier}} 
  \and Michael Wallner\affiliationmark{2}\thanks{\url{http://dmg.tuwien.ac.at/mwallner/}}
}

\title[Lattice paths with catastrophes]{Lattice paths with catastrophes}
\affiliation{
  LIPN UMR-CNRS 7030, Universit\'e de Paris Nord, France.\\
  Institute of Statistical Science, Academia Sinica, Taipei, Taiwan. 
}

\received{2016-7-10}
\accepted{2017-7-26}
\publicationdetails{19}{2017}{1}{23}{3776}

\usepackage{cmbright}
\usepackage[T1]{fontenc}
\usepackage{lmodern}

\usepackage{xcolor}
\usepackage{hyperref}
\definecolor{darkgreen}{rgb}{0,0.4,0}
\definecolor{BrickRed}{rgb}{0.65,0.08,0}
\hypersetup{colorlinks=true,linkcolor=blue,citecolor=darkgreen,filecolor=BrickRed,urlcolor=darkgreen}
\usepackage{amsmath,amssymb,color,amsthm}
\theoremstyle{definition}
\newtheorem{theorem}{Theorem}[section]
\newtheorem{definition}[theorem]{Definition}
\newtheorem{corollary}[theorem]{Corollary}
\newtheorem{proposition}[theorem]{Proposition}
\newtheorem{lemma}[theorem]{Lemma}
\newtheorem{remark}[theorem]{Remark}

\providecommand{\keywords}[1]{\textbf{\textit{Keywords: }} #1}

\usepackage{MnSymbol} 
\usepackage{subcaption} 
\usepackage{multirow} 
\usepackage{paralist} 
\usepackage{ifthen}
\usepackage{paralist}

\newcommand{\J}{\text{\em {J}}}

\newcommand{\LandauO}{O}

\newcommand{\PR}{\mathbb{P}} 
\newcommand{\E}{\mathbb{E}} 
 
\newcommand{\V}{\mathbb{V}} 
\newcommand{\Z}{\mathbb{Z}} 
\newcommand{\Bc}{\mathcal{B}}
\newcommand{\Dc}{\mathcal{D}}
\newcommand{\Ec}{\mathcal{E}}
\newcommand{\Fc}{\mathcal{F}}
\newcommand{\Hc}{\mathcal{H}}
\newcommand{\Mc}{\mathcal{M}}
\newcommand{\Nc}{\mathcal{N}}
\newcommand{\Qc}{\mathcal{Q}}
\newcommand{\Wc}{\mathcal{W}}
\newcommand{\Zc}{\mathcal{Z}}
\newcommand{\oeis}[1]{\text{\href{https://oeis.org/#1}{{\small \tt OEIS #1}}}} %https://oeis.org/A107083    A107083
\newcommand{\OEIS}[1]{\text{\href{https://oeis.org/#1}{{\small \tt #1}}}}      %https://oeis.org/A107083   (A107083)

\def\Acc{{\mathcal A}_{\textrm{cat}}}
\def\Anc{{\mathcal A}_{\textrm{nocat}}}
\def\Acgf{{A}_{\textrm{cat}}}
\def\Ancgf{{A}_{\textrm{nocat}}}
\newcommand{\Qcsing}{\eta}

\begin{document}
\maketitle

\begin{abstract}
In queuing theory, it is usual to have some models with a ``reset'' of the  queue.
In terms of lattice paths, it is like having the possibility of jumping from any altitude to zero.
These objects have the interesting feature that they do not have the same intuitive probabilistic behaviour 
as classical Dyck paths (the typical properties of which are strongly related to Brownian motion theory), 
and this article quantifies some relations between these two types of paths.
We give a bijection with some other lattice paths and a link with a continued fraction expansion. Furthermore, we prove several formulae for related combinatorial structures conjectured in the On-Line Encyclopedia of Integer Sequences.  
Thanks to the kernel method and via analytic combinatorics, 
we provide the enumeration and limit laws of these ``lattice paths with catastrophes'' for any finite set of jumps.
We end with an algorithm to generate such lattice paths uniformly at random.
\end{abstract}


\newpage
\section{Introduction}
\label{sec:intro}

Lattice paths are a natural model in queuing theory:
indeed, the evolution of a queue can be seen as a sum of jumps, a subject e.g.~considered in~\citet*{Feller68}.
In this article we consider jumps restricted to a given finite set of integers $\J$, 
where each jump $j\in \J$ is associated with a weight (or probability)~$p_j$.
The evolution of a queue naturally corresponds to lattice paths constrained to be non-negative.
For example, if $\J=\{-1,+1\}$, this corresponds to the so-called Dyck paths dear to the heart of combinatorialists.
Moreover,  we also consider the model where ``catastrophes'' are allowed.

\begin{definition}
\label{def:catastrophe}
A \emph{catastrophe} is 
a jump from an altitude $j>0$ ($-j \notin \J$) to altitude~$0$, see Figure~\ref{fig:catarchdecomp}. 
\end{definition}

Such a jump corresponds to a ``reset'' of the queue. 
The model of queues with catastrophes was e.g.~considered in~\citet*{Krinik05} and~\citet*{KrinikMohanty10}, which list many other references.
In financial mathematics, this also gives a natural, simple model of the evolution of stock markets, allowing bankruptcies at any time with a small probability $q$
(for the analysis and the applications of related discrete models, see e.g.~\citet*{Schoutens03, Elliott05}).
In probability theory and statistical mechanics, it was also considered under the name ``random walks with resetting'', 
see e.g.~\citet*{Majumdar14}. It is also related to random population dynamics, 
see e.g.~\citet*{BenAriRoitershteinSchinazi17}, to the random the Poland--Scheraga model for DNA denaturation,
as analysed by~\citet*{HarrisTouchette17}, or to Markov chains with restarts, as studied by~\citet*{JansonPeres12}. 
	\vspace{-2mm}
	\begin{figure}[ht]
		\begin{center}	
			\includegraphics[width=.8\textwidth]{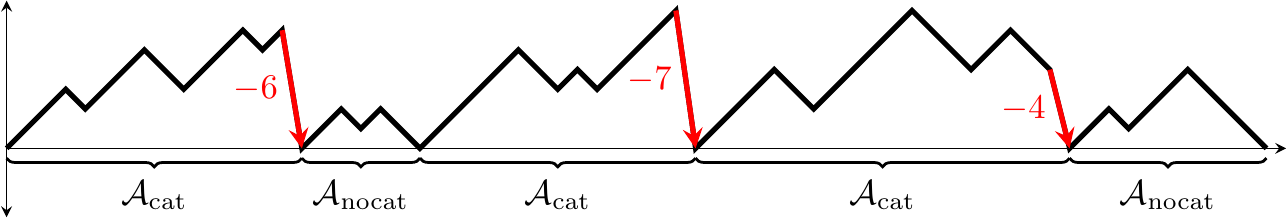}
			\caption{Decomposition of a Dyck path with $3$ catastrophes into 
			$5$ arches. $\Acc$ stands for an ``arch ending with a catastrophe''
(a walk for which the first return to altitude $0$ is a catastrophe), while $\Anc$ stands for an ``arch with no catastrophe''.}
			\label{fig:catarchdecomp}
		\end{center}
	\end{figure}
	\vspace{-4mm}

\textbf{Link with a continued fraction.}
We first start with the observation that the generating function $H(z)$ of Dyck paths with catastrophes ending at altitude $0$ has the following continued fraction expansion:
	\vspace{-5mm}
\begin{align}
	\label{eq:contfracH}
	H(z) = \cfrac{1}{1-\cfrac{z^2}{1-z-\cfrac{z^2}{1-\cfrac{z^2}{1-\cfrac{z^2}{1-\ddots}}}}}\,.
\end{align}
We give two proofs of this phenomenon in Theorem~\ref{theo:archbijection}.
In this article, we also tackle the question of what happens for more general sets of jumps than $\{-1,+1\}$,
and we provide the enumeration and asymptotics of the corresponding number of lattice paths under several constraints. 
\pagebreak

\smallskip
\textbf{Link with generating trees.}
In combinatorics, such lattice paths are related to generating trees, 
which are a convenient tool to enumerate and generate many combinatorial structures 
in some incremental way (like e.g.~permutations avoiding some pattern), see e.g.~\citet*{West96}. 
In such trees, the distribution of the children of each node follows exactly the same dynamics as lattice paths with some ``extended'' jumps,
as was intensively investigated by the Florentine school of combinatorics, e.g.~in~\citet*{Pinzani99,Fedou04,Rinaldi11}.
For example, these ``extended'' jumps can be a continuous set of jumps:
from altitude $k$, one can jump to any altitude between $0$ and $k$, possibly with some weights, plus a finite set of bounded jumps.
This model can be seen as an intermediate model between Dyck paths and our lattice paths with catastrophes;
we investigated it in our series of articles~\citet*{hexa,Banderier02,BanderierMerlini02,BanderierFedou03}. 
In this article, we will, however, see that several statistics of lattice paths with catastrophes 
behave in a rather different way than these walks with a continuous set of jumps,
even if they share the following unusual property:
both of them correspond to random walks with an ``infinite negative drift''
(in fact, a space-dependent drift tending to $-\infty$ when the altitude increases).  
At the same time, they are constrained to remain at non-negative altitudes; 
this leads to some counter-intuitive behaviour: 
unlike classical directed lattice paths, the limiting object is no more directly related to Brownian motion theory.

\smallskip
\textbf{Enumeration and asymptotics: why context-free grammars would be a wrong idea here.}
One way to analyse our lattice paths could be to use a context-free grammar approach, see~\citet*{LabelleYeh90}:
this leads to a system of algebraic equations, and therefore we already know ``for free'' that the corresponding generating functions are algebraic. 
However, this system involves nearly  $(c+d)^2$ equations (where $-c$ is the largest negative jump and $d$ the largest positive jump),
so solving it (with resultants or Gr\"obner bases) leads to computations taking a lot of time and memory (a bit complexity exponential in ${(c+d)^2}$):
even for $c=d=10$, the needed memory to compute the algebraic equation with this method would be more than the expected number of particles in the universe!
Another drawback of this method is that it would be a ``case-by-case'' analysis:
for each new set of jumps, one would have to do new computations from scratch. Hence, with this method, there is no way to access ``universal'' asymptotic results:
while it is well known that the coefficients of algebraic functions exhibit an asymptotic behaviour of the type $f_n \sim  C.A^n n^\alpha$,
only the ``critical exponent'' $\alpha$ can be proven to belong to a specific set (see~\citet*{BanderierDrmota15}). 
What is more, there is no hope to have easy access to $C$ and $A$ with this context-free grammar approach, in a way which is independent of a case-by-case computation
(which, what is more, would be impossible for $c+d>20$).

\smallskip
\textbf{The solution: kernel method and analytic combinatorics.}
In this article, we offer an alternative to context-free grammars. 
Our approach uses methods of analytic combinatorics for directed lattice paths: 
the kernel method and singularity analysis, as presented in~\citet*{BaFl02,flaj09}. It allows 
us to get exact enumeration, the typical behaviour of lattice paths with catastrophes,
and has the advantage of offering universal results for the asymptotics as well as generic closed forms, whatever the set of jumps is.

\bigskip

\textbf{ Plan of this article.} 
First, in Section~\ref{sec:genfunccat}, we present the model of walks with catastrophes and derive their generating functions. 
In Section~\ref{sec:catbij}, we establish a bijection between two generalizations of Dyck paths. 
In Section~\ref{sec:catasympt}, we analyse our model in more detail and first derive the asymptotic number of excursions and meanders. Then we use these results to obtain limit laws for the number of catastrophes, the number of returns to zero, the final altitude, the cumulative size of catastrophes, the average size of a catastrophe, and the waiting time for the first 
catastrophe (see Figures~\ref{fig:catarchparam} and~\ref{fig:cat}). 
In Section~\ref{sec:catuniformgeneration}, we discuss the uniform random generation of such lattice paths.
In Section~\ref{sec:catconclusion}, we summarize our results and mention some possible extensions.

\medskip

	\begin{figure}[ht]
		\begin{center}	
			\includegraphics[width=1\textwidth]{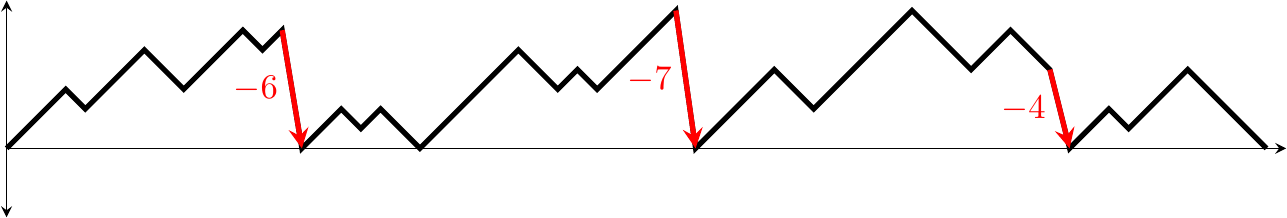}
			\caption{In this article, we analyse the number of Dyck paths with catastrophes, the waiting time for the first catastrophe (here, $15$, due to the $-6$ jump in red), its size (here, $6$),
 the number of returns to~$0$ (here, $5$), the number of catastrophes (here, $3$, in red), the total sum of their sizes (here, $6+7+4$), and therefore the average size of a catastrophe. 
			\label{fig:catarchparam}}
		\end{center}
	\end{figure}
	
	\begin{figure}[ht]
		\begin{center}	
			\includegraphics[width=1\textwidth]{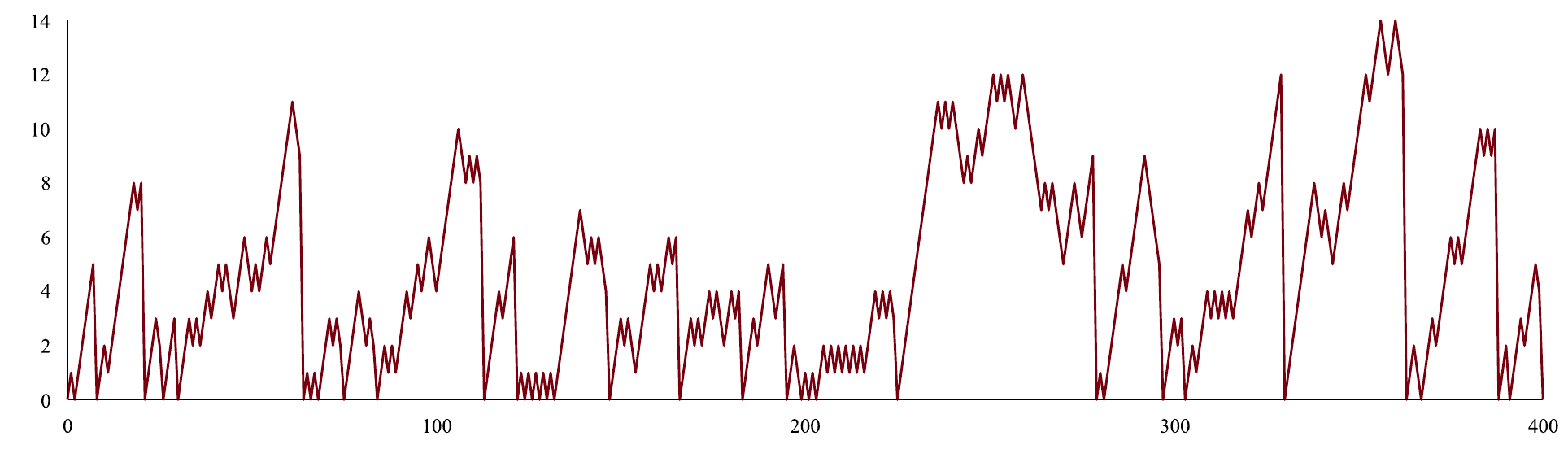}
			\caption{A Dyck path with catastrophes drawn uniformly at random (among excursions of length $n=400$). For this example, 
the walk has $35$ returns to zero, of which $22$ are catastrophes.
			The waiting time for the first catastrophe is $8$, its size is $5$, 
			the sum of the sizes of all catastrophes is $124$ (therefore an average catastrophe has size $5.63$).
For all these parameters, this article shows how these quantities evolve when $n$ gets larger.
More generally, we give the corresponding limit laws for walks with catastrophes allowing any finite set of jumps.
			\label{fig:cat}}
		\end{center}
	\end{figure}

\pagebreak

\section{Generating functions}
\label{sec:genfunccat}

In this section, we give some explicit formulae for the generating functions of non-negative lattice paths with catastrophes.
We consider the set of jumps $\{-c,\ldots,+d\}$ where a weight $p_i$ is attached to each jump $i$
and we associate to this set of jumps the following \emph{jump polynomial}:
\begin{align} \label{eqP}
	P(u)=\sum\limits_{i=-c}^d p_i u^i.
\end{align}
Every catastrophe is also assigned a weight $q > 0$. 
The weight of a lattice path is the product of the weights 
 of its jumps.
The weights $p_j$ and $q$ are real and non-negative: in fact, even if they would be taken from $\mathbb C$,  our {\em enumerative} formulae would remain valid.
The non-negativity of the weights or the fact that $d<+\infty$ and $P'(1)<+\infty$ 
only play a role for establishing the universal {\em asymptotic} phenomena presented in~Section~\ref{sec:catasympt}.

The generating functions of directed lattice paths can be expressed in terms of the roots $u_i(z)$, $i =1,\ldots,c$, 
of the \emph{kernel equation} 
\begin{equation}  \label{kerneleq} 1-zP(u_i(z))=0 \,.\end{equation} 
More precisely, this equation has $c+d$ solutions. 
The \emph{small roots} are the $c$ solutions with the property $u_i(z) \sim 0$ for $z \sim 0$. 
The remaining $d$ solutions are called \emph{large roots} as they satisfy $|v_i(z)| \sim +\infty$ for $z \sim 0$.
The generating functions of four classical types of lattice paths are shown in Table~\ref{tab:dirTypes}.
\begin{table}[h!b]
	\renewcommand{\arraystretch}{1.2}
	\begin{center}
	\begin{tabular}{|c|c|c|}
		\hline                        & ending anywhere & ending at $0$ \\
		\hline 
				                              & \multirow{3}{*}{\includegraphics[width=4.5cm]{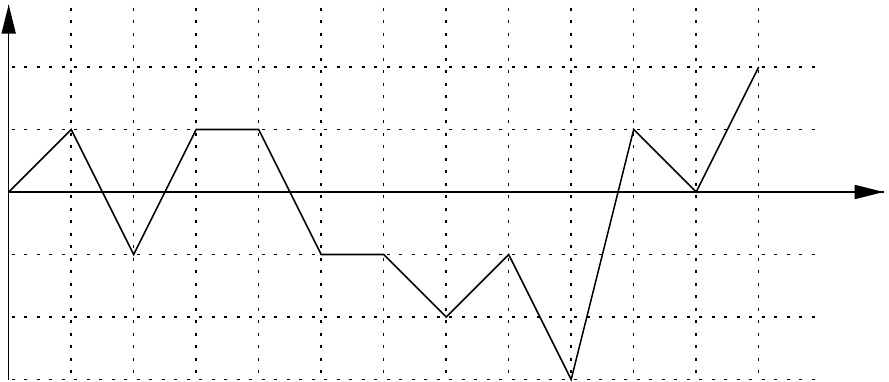}} & \multirow{3}{*}{\includegraphics[width=4.5cm]{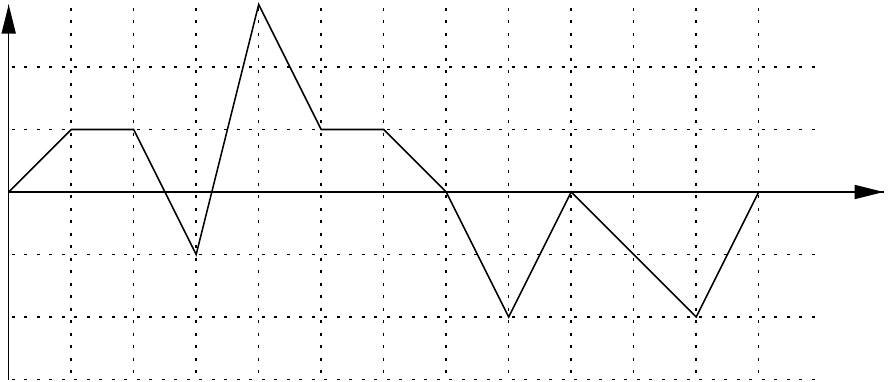}}\\
				                   & & \\
				                    unconstrained           & & \\
				                   (on $\Z$)& & \\
		                                  &  walk/path ($\Wc$)                   & bridge ($\Bc$)\\
		                                  &  $W(z) = \frac{1}{1-zP(1)}$          & $B(z) = z \sum\limits_{i=1}^c\frac{u_i'(z)}{u_i(z)}$\\		
	\hline 
				                   & \multirow{3}{*}{\includegraphics[width=4.5cm]{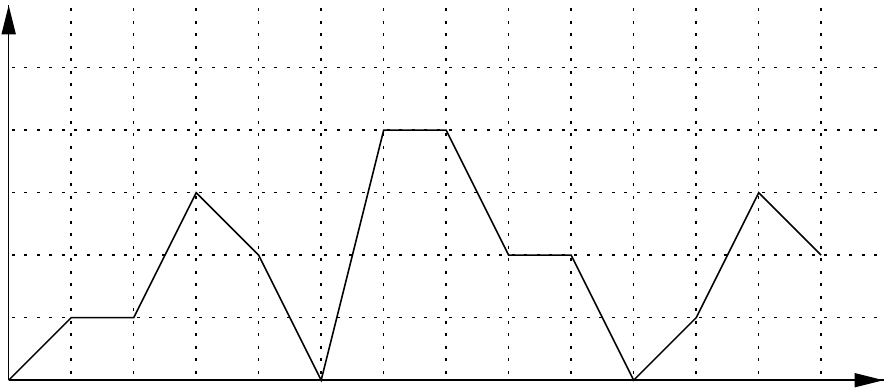}} & \multirow{3}{*}{\includegraphics[width=4.5cm]{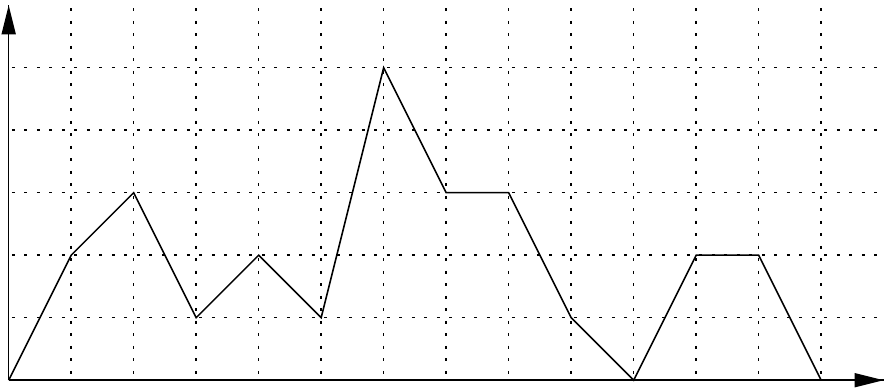}}\\
				                   & & \\
				               constrained    & & \\
				                   (on $\Z$)& & \\
		                             &  meander ($\Mc$)                       & excursion ($\Ec$)\\
		                                &  $M(z) = \frac{1}{1-zP(1)}\prod\limits_{i=1}^c(1-u_i(z))$          & $E(z) = \frac{(-1)^{c-1}}{p_{-c}z}\prod\limits_{i=1}^c u_i(z)$\\
		\hline
\end{tabular}
\end{center}
\caption{The four types of paths: walks, bridges, meanders and excursions, 
and the corresponding generating functions for directed lattice paths.}  
\label{tab:dirTypes}
\end{table}

\newpage
\smallskip                                                                                             
These results follow from the expression for the bivariate generating function $M(z,u)$ of meanders, see~\citet*{bope00} and~\citet*{BaFl02} :
Let $m_{n,k}$ be the number of meanders of length $n$ going from altitude $0$ to altitude $k$, then
\begin{align}	\label{eq:M}
	M(z,u) &= \sum_{n,k \geq 0} m_{n,k} z^n u^k = \sum_{k \geq 0} M_k(z) u^k = \frac{\prod_{i=1}^c(u-u_i(z))}{u^c(1-zP(u))}.
\end{align}
This formula is obtained by the kernel method: Starting from~\eqref{kerneleq}, it consists 
in setting $u=u_i(z)$ in the 
functional equation which mimics the recursive definition of a meander. This results in new and simpler equations which lead to the closed form~\eqref{eq:M}.
The generating function of excursions is $E(z):=M(z,0)$.

Let us now investigate which perturbation is introduced by allowing catastrophes in this model.
First, we partition the set of jumps $\J = \J_+ \cupdot \J_- \cupdot \J_0$ into the set of positive jumps 
($j \in \J_+$ iff\footnote{``iff'' is Paul Halmos' convenient  abbreviation of ``if and only if''.} 
$j > 0$), the set of negative jumps ($j \in \J_-$ iff $j<0$), and the possible zero jump ($j \in \J_0$ iff $j=0$).

\begin{theorem}[Generating functions for lattice paths with catastrophes]
\label{theo:LukaCatGF}
Let $f_{n,k}$ be the number of meanders with catastrophes of length $n$ from altitude $0$ to altitude $k$.
Then the generating function $F(z,u)=\sum_{k\geq 0} F_k(z) u^k= \sum_{n,k\geq 0} f_{n,k} u^k z^n$
is algebraic and satisfies
\begin{align}
	\label{eq:FandFkCat}
	F(z,u) &= D(z) M(z,u) 
	        = D(z) \frac{\prod_{i=1}^c(u-u_i(z))}{u^c(1-zP(u))}, \\
	F_k(z) &= D(z) M_k(z)
	        = D(z) \frac{1}{p_d z} \sum_{\ell = 1}^d v_{\ell}^{-k-1} \prod_{\substack{1\leq j\leq d\\ j \neq \ell}} \frac{1}{v_j - v_{\ell}}, \quad \text{ for } k \geq 0, \label{eq:Fk}
\end{align}
where $D(z) = \frac{1}{1 - Q(z)}$ is the generating function of excursions ending with a catastrophe, 
$Q(z) = zq \left( M(z) - E(z) - \sum_{-j \in \J_-} M_{j}(z) \right)$,
and where, for any set of jumps encoded by $P(u)$, 
the $u_i$'s and the $v_i$'s are the small roots and the large roots of the kernel equation~\eqref{kerneleq}.
\end{theorem}

\begin{proof}
Take an arbitrary non-negative path of length $n$. Let $\omega_0$ be the last time it returns to the $x$-axis with a catastrophe 
(or $\omega_0:=0$ if the path contains no catastrophe). This point gives a unique decomposition into an initial excursion which ends with a catastrophe (this might be empty), and a meander without any catastrophes. This directly gives~\eqref{eq:FandFkCat}.	

	What remains is to describe the initial part $D(z)$. 
	Consider an arbitrary excursion ending with a catastrophe. We decompose it with respect to its catastrophes, 
into a sequence of 
paths having only one catastrophe at their very end and none before, which we count by $Q(z)$. Thus,
	\begin{align}
		\label{eq:DseqE1}
		D(z) = \frac{1}{1-Q(z)}.
	\end{align} 
	
Because of Definition~\ref{def:catastrophe} of a catastrophe, $Q(z)$ is given by the generating function of meanders 
that are neither excursions nor meanders ending at altitudes $|j|$ ($j \in \J_-$) followed by a final catastrophe. 
This implies the shape of $Q(z)$: $Q(z) = zq \left( M(z) - E(z) - \sum_{-j \in \J_-} M_{j}(z) \right)$.
\end{proof}

\pagebreak

\begin{remark}\label{remarkQ}
	Our results 
	depend on the choice of Definition~\ref{def:catastrophe} of catastrophes. 
Some slightly different definitions could be used  without changing their structure.
For example, one could consider allowing catastrophes from any altitude (in this case, $Q(z) = zqM(z)$). 
In order to ensure an easy adaptation to different models, we will state all our subsequent results in terms of a generic $Q(z)$.
\end{remark}

Let us now consider a famous class of lattice paths (see e.g.~\citet*{Stanley99}),
which we call in this article ``classical'' Dyck paths. 
\begin{definition}
A \emph{Dyck meander} is a path constructed from the possible jumps~$+1$ and $-1$, each with weight~$1$, and being constrained to stay weakly above the $x$-axis. 
A \emph{Dyck excursion} is additionally constrained to return to the $x$-axis.
Accordingly, the polynomial encoding the allowed jumps is $P(u)=u^{-1}+u$. 
\end{definition}

For these paths, when one also allows catastrophes of weight $q=1$, one gets the following generating functions.

\begin{corollary}[Generating functions for Dyck paths with catastrophes]
	\label{coro:dyckmeandersandpaths}
	The generating function of Dyck meanders with catastrophes, $F(z,1)= \sum_{n\geq 0} m_{n} z^n$,
 satisfies
	\begin{align*}
		F(z,1) &= \frac{z(u_1(z)-1)}{z^2+(z^2+z-1)u_1(z)} 
		        =1+z+2z^2+4z^3+8z^4+17z^5+35z^6+\LandauO(z^7),
	\end{align*}
	where the small root $u_1(z)$ of the kernel is in fact the generating functions of Catalan numbers: 	$u_1(z) = \frac{1-\sqrt{1-4z^2}}{2}$.
	Moreover, $m_{n}$ is  also the number of equivalence classes of Dyck excursions of length $2n+2$ for the pattern \texttt{duu}, see \oeis{A274115}\footnote{Such references are 
	links to the web-page dedicated to the corresponding sequence in the On-Line Encyclopedia of Integer Sequences, \href{http://oeis.org}{http://oeis.org}.}.

	The generating function of Dyck excursions with catastrophes, $F_0(z) = \sum_{n\geq 0} e_{n} z^n$,
is 
	\begin{align*}
		F_0(z) &= \frac{(2z-1)u_1(z)}{z^2+(z^2+z-1)u_1(z)}=		1+z^2+z^3+3z^4+5z^5+12z^6+23z^7+\LandauO(z^8).
	\end{align*}
	This sequence $e_n$ corresponds to \oeis{A224747}.
	Moreover, $e_{2n}$ is  also the number of Dumont permutations of the first kind of length $2n$ avoiding the patterns 1423 and 4132, see \oeis{A125187}.
\end{corollary}
\begin{proof} 
The formulae for $F(z,1)$ and $F_0(z)$ are a direct application of Theorem~\ref{theo:LukaCatGF}. 
Then one notes that  $(F_0(z)+F_0(-z))/2$ equals the generating function of Dumont permutations of the first kind of length $2n$ avoiding the patterns 1423 and 4132,
see~\citet*{Burstein05} for the definition of such permutations, and the derivation of their generating function.
In~\citet*{Sapounakis16}, two Dyck excursions are said to be equivalent if they have the same length and all occurrences of the pattern \texttt{duu} are at the same places. 
They derived the generating function for the number of equivalence classes, which appears to be equal to $F(z,1)$.
\end{proof}

In the next section we will analyse Dyck paths with catastrophes in more detail. 
On the way we solve some conjectures of the On-Line Encyclopedia of Integer Sequences.

\pagebreak
\section{Bijection for Dyck paths with catastrophes}
\label{sec:catbij}

The goal of this section is to establish a bijection between two classes of extensions of Dyck paths\footnote{We thus prove several conjectures by Alois P.~Heinz,  R.~J.~Mathar,  and other contributors in the On-Line Encyclopedia of Integer Sequences, see sequences \OEIS{A224747} and \OEIS{A125187} therein.}. 
We consider two extensions of classical Dyck paths (see Figure~\ref{fig:archbijection} for an illustration):
\begin{enumerate}
	\item \emph{Dyck paths with catastrophes} are Dyck paths with the additional option of jumping to the $x$-axis from any altitude $h > 1$; and
	\item \emph{$1$-horizontal Dyck paths} are Dyck paths with the additional allowed horizontal step $(1,0)$ at altitude $1$. 
\end{enumerate}
\begin{figure}[ht]
 \begin{center}
	\begin{subfigure}[b]{0.45\textwidth}
		\centering 
		\includegraphics[width=1\textwidth]{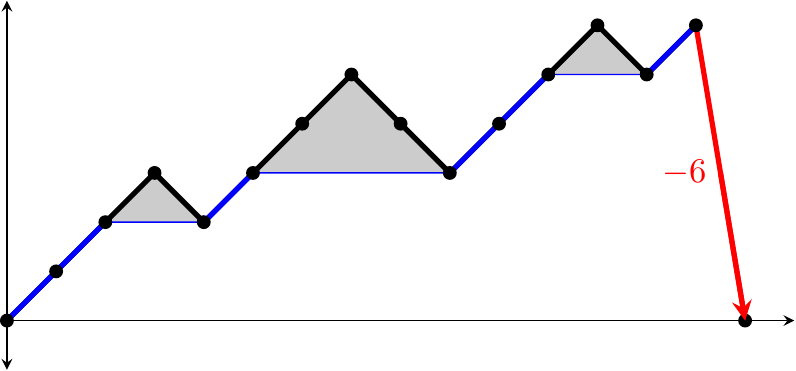}
		\caption{Dyck arch ending with a catastrophe}
	\end{subfigure}
	\qquad
	\begin{subfigure}[b]{0.45\textwidth}
		\centering 
		\includegraphics[width=1\textwidth]{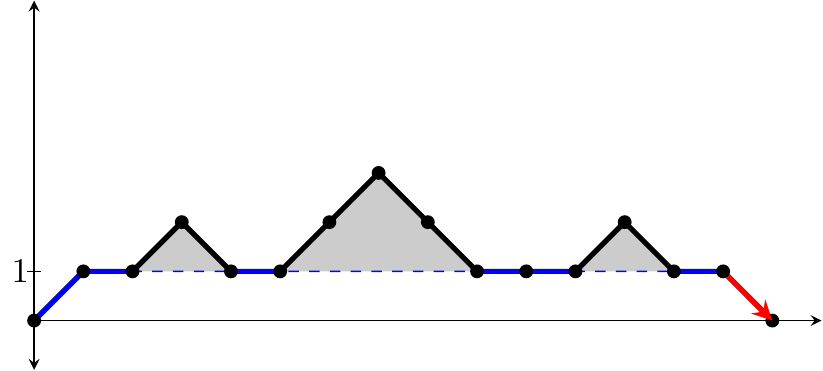}
		\caption{$1$-horizontal Dyck arch}
	\end{subfigure}
	\caption{
	An arch is an excursion going back to altitude~$0$ only once.
	The bijection of Theorem~\ref{theo:archbijection} transforms Dyck arches ending with catastrophes into $1$-horizontal Dyck arches, and vice versa.}
	\label{fig:archbijection}
 \end{center}
\end{figure}
\vspace{-5mm}
\begin{theorem}[Bijection for Dyck paths with catastrophes] The number $e_n$ of Dyck paths with catastrophes of length $n$ is equal 
to the number $h_n$ of $1$-horizontal Dyck paths of length $n$:  	$e_n = h_n.$
	\label{theo:archbijection}
\end{theorem}
\begin{proof}
A first proof that $e_n=h_n$ follows from a continued fraction approach: Each level~$k+1$ of the continued fraction encodes the non-negative jumps starting from altitude~$k$ and the negative jumps going to altitude $k$ (see e.g.~\citet*{Flajolet80}). The jumps $+1$ and $-1$ translate into $z^2$ 
in Equation~\eqref{H_CF} below. At altitude $1$, a horizontal jump is also allowed; this translates into the $z$ term in the continuous fraction.  We thus get the continued fraction of Equation~\eqref{eq:contfracH}. Simplifying its periodic part, we get
\begin{equation}\label{H_CF}
H(z) = \sum_{n\geq 0} h_n z^n = \cfrac{1}{1-\cfrac{z^2}{1-z-z^2 C(z)}}\,,
\end{equation}
where $C(z)$ is the generating function of classical Dyck paths, $C(z)=1/(1-z^2C(z))$. One then gets that $H(z)$ equals the closed form
of $F_0(z)$ given in Corollary~\ref{coro:dyckmeandersandpaths}. 
We now give a bijective procedure which transforms every Dyck path with catastrophes into a $1$-horizontal Dyck path, and vice versa. 

	Every Dyck path with catastrophes can be decomposed into a sequence 
of	arches, see Figure \ref{fig:catarchdecomp}. 
	There are two types of arches: arches ending with catastrophes $\Acgf(z)$ and arches ending with a jump $j \in \J$ given by $\Ancgf(z)$. This gives the alternative decomposition to~\eqref{eq:Fk} 
	illustrated in Figure~\ref{fig:catarchdecomp}:
	\begin{align*}
		F_0(z) &= \frac{1}{1-(\Acgf(z) + \Ancgf(z))}.
	\end{align*} 
	Thus, without loss of generality, we continue our discussion only for arches. The following procedure is visualized in Figure \ref{fig:archbijection}.
	
	Let us start with an arbitrary arch of Dyck paths with catastrophes. It is either a classical Dyck path, and therefore also a $1$-horizontal Dyck path, or it ends with a catastrophe of size $h$. 
	First, we associate the catastrophe with $h$ up steps $(1,1)$. Specifically, we draw horizontal lines to the left until we hit an up step. All but the first one are replaced by horizontal steps. Finally, we replace the catastrophe by a down step $(1,-1)$. All parts in between stay the same. Note that we replaced $h-1$ up steps, and therefore decreased the altitude by $h-1$, but we also replaced the catastrophe of size~$h$ by a down step, which represents a gain of altitude by $h-1$. Thus, we again return to the $x$-axis. Furthermore, all horizontal steps are at altitude~$1$.
	Thus, we  always stay weakly above the $x$-axis, and we got an arch of a $1$-horizontal Dyck path.
	The inverse mapping is analogous. 
\end{proof}

The most important building blocks in the previous bijection were arches ending with a catastrophe. 
Let us note that these can be enumerated by an explicit formula.

\begin{proposition}[Dyck arches ending with a catastrophe]
	\label{prop:arches}
	Let $A(z) = \sum_{n \geq 0} a_n z^n$ be the generating function of arches ending with a catastrophe. Then one has the following closed forms
	\begin{align*}
			a_{n} &= \binom{n-2}{\lfloor \frac{n-3}{2} \rfloor},\\
		\Acgf(z) &= z \frac{M(z) - E(z) - M_1(z)}{E(z)} = \frac{1}{2} \, {\frac {2 z^2+z-1+\sqrt {\left( 1-2 z \right)  \left( 1+ 2 z \right)  \left( 1-z \right) ^{2}}}{1-2 z}}\\
&=		z^3+z^4+3 z^5+4 z^6+10 z^7+15 z^8+35 z^9+\LandauO(z^{10}),
	\end{align*}
	where $M(z), E(z)$ and $M_1(z)$ are the generating functions of classical Dyck paths for meanders, excursions, and meanders ending at~$1$, respectively, see \oeis{A037952}.	
\end{proposition}

\begin{proof}
	Every excursion ending with a catastrophe can be uniquely decomposed into an initial excursion and a final arch with a catastrophe.
	By Theorem~\ref{theo:LukaCatGF} we get the generating function of $\Acgf(z) = \frac{Q(z)}{E(z)}$. 
	
	In order to compute $a_{n}$, additionally drop the initial up-jump which is necessary for all such arches of positive length. The remaining part is a Dyck meander (always staying weakly above the $x$-axis) that does not end on the $x$-axis. Thus,
	\begin{align*}
		a_{n+2} &= \underbrace{\binom{n}{\lfloor \frac{n}{2} \rfloor}}_{\text{meanders}} - \underbrace{\frac{1}{n/2+1}  \binom{n}{\frac{n}{2}}}_{\text{excursions}} [\![ n \text{ even} ]\!],
	\end{align*}
	where $[\![ P ]\!]$ denotes the Iverson bracket, which is $1$ if the condition $P$ is true, and $0$ otherwise. 
\end{proof}

\section{Asymptotics and limit laws}
\label{sec:catasympt}

The natural model in which all paths of length $n$ have the same weight
creates a probabilistic model 
in which
the drift of the walk is then space-dependent: it converges to minus infinity when the altitude of the paths is increasing.  
So, unlike the easier classical Dyck paths (and their generalization via directed lattice paths, 
having a finite set of given jumps), we are losing the intuition offered by Brownian motion theory.
This leads to the natural question of what the asymptotics of the fundamental parameters 
of our ``lattice paths with catastrophes'' are.   This is the question we are going to answer now.

To this aim, some useful results of \citet*{BaFl02} are the asymptotic enumeration formulae 
for the four types of paths shown in Table~\ref{tab:dirTypes}. 
A key result is the fact that the principal small root~$u_1(z)$ and the principal large root~$v_1(z)$ 
of the kernel equation~\eqref{kerneleq}
are conjugated to each other at their dominant singularity~$\rho=\frac{1}{P(\tau)}$, where $\tau>0$ is the minimal real positive solution of $P'(\tau)=0$. In particular, it holds that
\begin{align*}
	u_1(z) &= \tau - C\sqrt{1-z/\rho} + \LandauO\left(1-z/\rho\right),\\
	v_1(z) &= \tau + C\sqrt{1-z/\rho} + \LandauO\left(1-z/\rho\right),
\end{align*}
with the constant $C := \sqrt{2\frac{P(\tau)}{P''(\tau)}}$. This singularity $\rho$ of $u_1(z)$ and $v_1(z)$ turns out to be the dominant singularity of the generating functions of directed lattice paths.

We refer to~\citet*{flaj09} for the notion of dominant singularity 
and a clear presentation of its fundamental role in the asymptotics of the coefficients of a generating function.
It could be the case that the generating functions have several dominant singularities.
This happens when one has a periodicity in the support of the jump polynomial $P$.
\begin{definition}
	We say that a function $F(z)$ has periodic support of period $p$ or (for short) $F(z)$ is \emph{$p$-periodic} if there exists a function $H(z)$ and an integer $b$ such that $F(z)=z^b H(z^p)$. If this holds only for $p=1$, the function is said to be \emph{aperiodic}.
\end{definition}

In \citet*[Lemma~8.7 and Theorem~8.8]{BanderierWallner16} we show how to deduce the asymptotics of walks having periodic jump polynomials from the results on aperiodic ones. 
Therefore, without loss of generality we consider only the aperiodic jump polynomials in this article. 

\subsection{Asymptotic number of lattice paths}
Because of its key role in the expression given in Theorem~\ref{theo:LukaCatGF},
we start by analysing the function $D(z) = \frac{1}{1-Q(z)}$. 
In particular, we need to find its singularities, which are given by the behaviour of $$Q(z) = zq \left( M(z) - E(z) - \sum_{-j \in \J_-} M_{j}(z) \right)\,.$$
\pagebreak

Caveat: Even if we already know that the radii of convergence $\rho_M, \rho_E,\rho_{M_{j}}$ of $M(z), E(z), M_{j}(z)$, it is a priori not granted that $Q(z)$ does not have a larger radius of convergence (some cancellations could occur).
In fact, the results of~\citet*{BaFl02} allow us to prove that no such cancellations occur here via the asymptotics of the coefficients of $M(z), E(z),$ and $M_{j}(z)$. Therefore the radius of convergence of $Q(z)$ is $\rho_Q = \min(\rho_M,\rho_E,\rho_{M_{j}})$. 

\smallskip
We now determine the radius of convergence $\rho_D$ of $D(z)$. 

\begin{lemma}[Radius of convergence  of $D$]
	\label{lem:Ddenom}
	Let $\Zc$ be the set of zeros of $1-Q(z)$ of minimal modulus with $|z|\leq \rho$.
This set is either empty or has exactly one real positive element which we call $\rho_0$.
The sign of the drift $\delta :=P'(1)$ of the walk dictates the location of the radius of convergence $\rho_D$: 
	\begin{itemize}
		\item \text{If $\delta \geq 0$, we have $\rho_D=\rho_0<\frac{1}{P(1)} \leq \rho$.}\\
		 \item \text{If $\delta<0$, it also depends on the value $Q(\rho)$:}\\
		\begin{align*}
		\begin{cases}
		 Q(\rho)>1 \quad \Longleftrightarrow \quad \rho_D=\rho_0 < \rho,\\
       		Q(\rho)=1 \quad \Longleftrightarrow \quad \rho_D=\rho_0 = \rho, \\
		  Q(\rho)<1 \quad \Longleftrightarrow \quad \rho_D= \rho  \text{\quad and $\Zc$ is empty.}
		\end{cases}
		\end{align*}
		\end{itemize}
\end{lemma}

\begin{proof}
	As $D(z)=1/(1-Q(z))$ is a generating function with positive coefficients, Pringsheim's theorem implies that
	it has a dominant singularity on the real positive axis, which we call $r$.
	This singularity is either $r=\rho_Q$, the singularity of $Q(z)$, or it is the smallest real positive zero of $1-Q(z)$ (if it exists, it is denoted by $\rho_0$ and it is therefore such that $\rho_0 \leq \rho_Q$).
	  
	What is more, $D(z)$ has no other dominant singularity:
	When $r=\rho_Q$ this follows from the aperiodicity of $E(z), M(z)$, and $M_{j}(z)$ 
proven in \citet*{BaFl02}. When $r = \rho_0$ this follows from the strong triangle inequality. Indeed, as $Q(z)$ has non-negative coefficients and is aperiodic, one has $|Q(z)| < Q(|z|) = 1$ for any $|z| = r$, $z \neq r$.
	 
It remains to determine the location of the singularity.
The functions $E(z)$ and $M_{j}(z)$ are analytic for $|z|<\rho$, whereas the behaviour of $M(z)$ depends on the drift $\delta = P'(1)$. 
For $\delta \geq 0$ it possesses a simple pole at $\rho_1 := \frac{1}{P(1)} \leq \rho$, i.e.~$\rho_Q = \rho_1$. Thus, $\lim_{z \to \rho_1^-} Q(z) = +\infty$, and together with $Q(0)=0$ this implies that there is a solution $0 < \rho_0 < \rho_1 \leq \rho$. 
	
	For $\delta < 0$ we have that $|Q(z)|$ is bounded for $|z|<\rho$, and we have $\rho_Q = \rho$. Thus, for a fixed jump polynomial $P(u)$ any case can be attained by a variation of $q$. As $Q(z)$ is monotonically increasing on the real axis, it suffices to compare its value at its maximum $Q(\rho)$.
\end{proof}

Note that $Q(z)$ strongly depends on the weight of the catastrophes $q>0$. Therefore, for a fixed step set $P(u)$ with {\em negative} drift one can obtain any of the three possible cases by a proper choice of~$q$.

\pagebreak
\begin{theorem}[Asymptotics of excursions ending with a catastrophe]
	\label{theo:Dasym}
	Let $d_n$ be the number of excursions ending with a catastrophe. The asymptotics of~$d_n$ depend on the structural radius $\rho$ and the possible singularity $\rho_0$: 
	\begin{align*}
		d_n = 
			\begin{cases}
				\frac{\rho_0^{-n}}{\rho_0 Q'(\rho_0)} + o(K^n)  & \text{\quad if } \rho_0 < \rho  \text{\quad (for some $K$ such that $K<\frac{1}{\rho_0}$)},\\[3mm]
				\frac{\rho^{-n}}{\Qcsing \sqrt{\pi n}}  \left( 1 + \LandauO\left(\frac{1}{n}\right) \right) & \text{\quad if } \rho_0 = \rho,\\[3mm]
				 \frac{D(\rho)^2 \Qcsing \rho^{-n}}{2 \sqrt{\pi n^3}}  \left( 1 + \LandauO\left(\frac{1}{n}\right) \right) & \text{\quad if $\rho_0$ does not exist},\\
			\end{cases}
	\end{align*}
	where  $\Qcsing$ is given by the Puiseux expansion of $Q(z) = Q(\rho) - \Qcsing \sqrt{1-z/\rho} + \LandauO(1-z/\rho)$ for $z \to \rho$. The last two cases occur only when $\delta <0$.
\end{theorem}

\begin{proof} 
	The critical exponent $\alpha$ in the Puiseux expansion of $D(z)$ 
	in $(1-z/r)^\alpha$ for $r=\rho_0$ or $r=\rho$, respectively, satisfies
	\begin{compactitem}
		\item $\alpha=-1$ if $\rho_0 < \rho$,
		\item $\alpha=-1/2$ if $\rho_0 = \rho$,
		\item $\alpha=1/2$ if $\rho_0$ does not exist.
			\end{compactitem}
Indeed,
	the singularity of $D(z)$ arises at the minimum of $\rho$ and $\rho_0$, as derived in Lemma~\ref{lem:Ddenom}. 
	In the first case $\rho_0<\rho$,  the singularity is a simple pole as the first derivative of the denominator at $\rho_0$ is strictly positive. We get for $z \to \rho_0$
	\begin{align}
		\label{eq:Qleadvanish}
		1-Q(z) &= \underbrace{(1- Q(\rho_0))}_{=0} + \rho_0 Q'(\rho_0)) (1-z/\rho_0) + \LandauO\left( (1-z/\rho_0)^2 \right).
	\end{align}
	This yields a simple pole at $\rho_0$ for $\frac{1}{1-Q(z)}$ and singularity analysis then gives the asymptotics of $d_n$.
	
	If $\rho_0$ does not exist, or if $\rho_0 = \rho$,
	we get a square root behaviour for $z \to \rho$
	\begin{align}
		\label{eq:Drecexpanded}
		1-Q(z) &= (1- Q(\rho)) + \Qcsing \sqrt{1-z/\rho} + \LandauO\left( 1-z/\rho \right).
	\end{align}	
	
	For $\rho_0 = \rho$ the constant term is $0$, and we get for $z \to \rho$
	\begin{align}
		\label{eq:Drho0expanded}
		D(z) &= \frac{1}{\Qcsing \sqrt{1-z/\rho}} \left( 1 + \LandauO(\sqrt{1-z/\rho}) \right).
	\end{align}
	
	Yet, if $\rho_0$ does not exist, the constant term does not vanish. This gives for $z \to \rho$
	\begin{align}
		\label{eq:Dexp}
		D(z) &= D(\rho) - \Qcsing D(\rho)^2 \sqrt{1-z/\rho}  + \LandauO\left( 1-z/\rho \right).
	\end{align}	
	Applying singularity analysis yields the result.
\end{proof}

With the help of the last result we are able to derive the asymptotic number of lattice paths with catastrophes. Let us state the result for excursions next.
\pagebreak
\begin{theorem}[Asymptotics of excursions with catastrophes]
	\label{theo:exc}
	The number of excursions with catastrophes $e_n$ is asymptotically equal to
	\begin{align*}
		e_n &= 
			\begin{cases}
				 \frac{E(\rho_0) }{\rho_0 Q'(\rho_0)} \rho_0^{-n}  + o(K^n) \text{\quad (for some $K$ such $K<1/\rho_0$)} & \text{\quad if } \rho_0<\rho, \\[3mm]
				 \frac{E(\rho) }{\Qcsing } \frac{\rho^{-n}}{ \sqrt{\pi n}}  \left(1 + \LandauO\left(\frac{1}{n}\right)\right) & \text{\quad if }  \rho_0 = \rho, \\[3mm]
				\frac{F_0(\rho)}{2} \left(\sqrt{2 \frac{P(\tau)}{P''(\tau)}} \frac{1}{\tau} + \Qcsing D(\rho) \right) \frac{\rho^{-n}}{\sqrt{\pi n^3}} \left( 1 + \LandauO\left(\frac{1}{n}\right) \right) & \text{\quad if $\rho_0$ does not exist}.
			\end{cases}
	\end{align*}	
	\vspace{-2mm}
\end{theorem}

\begin{proof}
	Since $F_0(z) = D(z) E(z)$, the singularity is either at $\rho_0$ or $\rho = \frac{1}{P(\tau)}$. Combining the results from Theorem~\ref{theo:Dasym} and \citet*[Theorem~3]{BaFl02} gives the result. Note that the cases $\rho_0 = \rho$ and when $\rho_0$ does not exist are only possible for $\delta<0$.
\end{proof}

Next we also state the asymptotic number of meanders. The only difference is the appearance of $M(z)$ instead of $E(z)$, 
and  a factor $\frac{1}{\tau-1}$ instead of $\frac{1}{\tau}$ in the first term when $\rho_0$ does not exist.

\begin{theorem}[Asymptotics of meanders with catastrophes]
	\label{theo:mea}
	The number of meanders with catastrophes $m_n$ is asymptotically equal to
	\begin{align*}
		m_n &= 
			\begin{cases}
				 \frac{M(\rho_0) }{\rho_0 Q'(\rho_0)} \rho_0^{-n} +o(K^n) \text{\quad (for some $K$ such $K<1/\rho_0$)} & \text{\quad if }  \rho_0<\rho, \\[2mm]
				 \frac{M(\rho) }{\Qcsing } \frac{\rho^{-n}}{ \sqrt{\pi n}}  \left(1 + \LandauO\left(\frac{1}{n}\right)\right) & \text{\quad if } \rho_0 = \rho, \\[2mm]
				\frac{F(\rho,1)}{2} \left(\sqrt{2 \frac{P(\tau)}{P''(\tau)}} \frac{1}{\tau-1} + \Qcsing D(\rho) \right) \frac{\rho^{-n}}{\sqrt{\pi n^3}} \left( 1 + \LandauO\left(\frac{1}{n}\right) \right) & \text{\quad if $\rho_0$ does not exist}.
			\end{cases}
	\end{align*}
\end{theorem}

\begin{proof}
	Analogous to the proof of Theorem~\ref{theo:exc} the result follows after some tiresome computations from the fact that $F(z,1) = D(z) M(z)$. 
	Combining the results from Theorem~\ref{theo:Dasym} and \citet*[Theorem~4]{BaFl02} gives the result.
\end{proof}

\vspace{-2mm}

\begin{remark}
	In the previous proofs we needed that $P(u)$ is an aperiodic jump set. Otherwise, the generating function $Q(z)$ does not have a unique singularity on its circle of convergence, but several. In such cases one needs to consider all singularities and sum their contributions; however, this can lead to cancellations, 
	thus extra care is necessary. A systematic approach of the periodic cases is treated in~\citet*{BanderierWallner16}.	
These considerations about periodicity are only necessary when the dominant asymptotics come from the singularity $\rho$,
while when $\rho_0 < \rho$,  we have a unique dominant simple pole (the possibly periodic functions $E(z)$ and $M_{j}(z)$ do not contribute to the asymptotics).
This polar behaviour occurs e.g.~for Dyck paths.
\end{remark}

\begin{corollary}
	\label{prop:Dyckpathasy}
	The number of Dyck paths with catastrophes~$e_n$ and Dyck meanders with catastrophes~$m_n$ are 	respectively asymptotically equal to
	\vspace{-2mm}
	\begin{align*}
		e_n &= C_e \rho_0^{-n} \left(1 + \LandauO\left(1/n\right)\right) \text{\quad ($C_e \approx 0.10381$ is the positive root of $31C_e^3-62C_e^2+35C_e-3$) \text{~and}}\\
		m_n &= C_m \rho_0^{-n} \left(1 + \LandauO\left(1/n\right)\right) 
\text{\quad	($C_m \approx 0.32679$ is the positive root of $31C_m^3-31C_m^2+16C_m-3$),}
	\end{align*}
	where 
	$\rho_0 \approx 0.46557$ 
	is the unique positive root of $\rho_0^3+2 \rho_0^2+\rho_0-1$.\\
	
\end{corollary}\pagebreak
\begin{proof}
	We apply the results of Theorems~\ref{theo:exc} and \ref{theo:mea}. This directly gives $$\rho_0 = \frac{1}{6} \left(116 + 12 \sqrt{93}\right)^{1/3} + \frac{2}{3} \left(116 + 12 \sqrt{93}\right)^{-1/3} - \frac{2}{3} \approx 0.46557,$$ which is strictly smaller than $\rho = 1/2$. The minimal polynomials of the algebraic numbers $\rho_0, C_e$, and~$C_m$ are computed by resultants.
\end{proof}

\begin{remark}
	It is one of the surprising behaviours of Dyck paths with catastrophes: they involve algebraic quantities of degree~$3$; this was quite counter-intuitive to predict a priori, as Dyck path statistics usually involve by design algebraic quantities of degree~$2$.
\end{remark}

As a direct consequence of the last two theorems,
we observe that our walks with catastrophes 
have the feature that excursions and meanders have the same order of magnitude: $e_n = \Theta(m_n)$,
whereas it is often $e_n = \Theta(m_n/n)$ for other classical models of lattice paths.
In probabilistic terms this means that the set of excursions is not a null set with respect to the set of meanders. 
We quantify more formally this claim in the following corollary.

\begin{corollary}[Ratio of excursions]
The probability that a meander with catastrophes of length $n$ is in fact an excursion is equal to
	\begin{align*}
	\frac{e_n}{m_n} \sim
&
			\begin{cases}
				 \frac{E(\rho_0) }{M(\rho_0)}  & \text{\quad if }  \rho_0<\rho, \\[3mm]
				 \frac{E(\rho) }{M(\rho)} & \text{\quad if } \rho_0 = \rho, \\[3mm]
				\frac{F_0(\rho)}{F(\rho,1)}  
				\left(1- 1/\left( \tau +  \sqrt{\frac{P''(\tau)}{2 P(\tau)}} \tau (\tau-1) \eta  D(\rho)\right)\right)
				   & \text{\quad if $\rho_0$ does not exist}.
			\end{cases}
	\end{align*}
For Dyck walks, this gives $\PR$(meander of length $n$ is an excursion)~$=e_n/m_n \approx 0.31767 $.
\end{corollary}

In the next sections, we will need the following variant of the supercritical composition scheme from \citet*[Proposition IX.6]{flaj09}, in which we add a perturbation function $q(z)$. In the following, we denote by $\rho_f$ the radius of convergence of a function $f(z)$.

\begin{proposition}[Perturbed supercritical composition]
	\label{prop:persupcomp}
	Consider a combinatorial structure $\Fc$ constructed from $\Hc$ components according to the bivariate composition scheme $F(z,u) = q(z) g(uh(z))$. 
Assume that $g(z)$ and $h(z)$ satisfy the supercriticality condition $h(\rho_h) > \rho_g$, that $g$ is analytic in $|z|<R$ for some $R > \rho_g$, with a unique dominant singularity at $\rho_g$, which is a simple pole, and that $h$ is aperiodic. 
	Furthermore, let $q(z)$ be analytic for $|z| < \rho_h$. 	
	Then the number $\chi$ of $\Hc$-components in a random $\Fc$-structure of size $n$, corresponding to the probability distribution $[u^k z^n] F(z,u)/[z^n]F(z,1)$ has a mean and variance that are asymptotically proportional to $n$; after standardization, the parameter $\chi$ satisfies a limiting Gaussian distribution, with speed of convergence $\LandauO(1/\sqrt{n})$. 
\end{proposition}

\begin{proof}
As $q(z)$ is analytic at the dominant singularity, it contributes only a constant factor to the asymptotics. 
Then Hwang's quasi-power theorem, 
see~\citet*[Theorem IX.8]{flaj09}, gives the claim.
\end{proof}

A simple (and useful) application of this result in the context of sequences leads to: 

\begin{proposition}[Perturbed supercritical sequences]
	\label{prop:pertsupcritseq}
	Consider a sequence scheme $\Fc = \Qc \times \operatorname{Seq}(u \Hc)$ that is supercritical, i.e., the value of $h$ at its dominant positive singularity satisfies $h(\rho_h) >1$. Assume that $h$ is aperiodic, $h(0)=0$, and $q(z)$ is analytic for $|z|<\rho$, where $\rho$ is the positive root of $h(\rho)=1$. Then the number $X_n$ of $\Hc$-components in a random $\Fc$-structure of size $n$ is, after standardization, asymptotically Gaussian with%
	\footnote{The formula for the asymptotics of $\V(X_n)$ in {\citet*[Proposition~IX.7]{flaj09}}  contains some typos and misses the $\rho$-factors in the numerator and one in the denominator.}
	\begin{align*}
		\E(X_n) &\sim \frac{n}{\rho h'(\rho)}, & \V(X_n) &\sim n \frac{\rho h''(\rho)+h'(\rho)-\rho h'(\rho)^2}{\rho^2 h'(\rho)^3}.
	\end{align*}
	What is more,
	the number $X_n^{(m)}$ of components of some fixed size $m$ is asymptotically Gaussian with asymptotic mean~$\sim~\theta_m n$, where $\theta_m = h_m \rho^m / (\rho h'(\rho))$. 
\end{proposition}

\begin{proof}
	The proof follows exactly the same lines as \citet*[Proposition IX.7]{flaj09}. We state it for completeness. 
	The first part is a direct consequence of Proposition~\ref{prop:persupcomp} with $g(z) = (1-z)^{-1}$ and  $\rho_g$ replaced by $1$. The second part results from the bivariate generating function 
	\begin{align*}
		F(z,u) &= \frac{q(z)}{1-(u-1) h_mz^m - h(z)},
	\end{align*}
	and from the fact that $u$ close to $1$ induces a smooth perturbation of the pole of $F(z,1)$ at $\rho$, corresponding to $u=1$.
\end{proof}

\subsection{Average number of catastrophes}

In Theorem~\ref{theo:LukaCatGF} we have seen that excursions consist of two parts: a prefix containing all catastrophes followed by the type of path one is interested in. If we want to count the number of catastrophes, it suffices therefore to analyse this prefix given by $D(z)$. What is more, due to~\eqref{eq:DseqE1} we already know how to count catastrophes: by counting occurrences of $Q(z)$. Thus, let $d_{n,k}$ be the number of excursions ending with a catastrophe of length $n$ with $k$ catastrophes. Then we have
\begin{align*}
	D(z,v) := \sum_{n,k \geq 0} d_{n,k} z^n v^k = \frac{1}{1-v Q(z)}.
\end{align*}

Let $c_{n,k}$ be the number of excursions with $k$ catastrophes, we get
\begin{align}
	\label{eq:gfnrcat}
	C(z,v) := \sum_{n,k \geq 0} c_{n,k} z^n v^k = D(z,v) E(z).
\end{align}

Let $X_n$ be the random variable 
giving the number of catastrophes in excursions of length $n$ drawn uniformly at random:
\begin{align*}
	\PR\left(X_n = k\right) &= \frac{[z^nv^k]C(z,v)}{[z^n]C(z,1)}.
\end{align*} \vspace{1mm}

\begin{theorem}[Limit law for the number of catastrophes]
	\label{theo:limitcat2}
	The number of catastrophes of a random excursion with catastrophes of length $n$ admits a limit distribution, with the limit law being dictated by the relation between the singularities $\rho$ (the structural radius $\rho=1/P(\tau)$ where $\tau>0$ is the minimal real positive solution of $P'(\tau)=0$) and $\rho_0$ (the minimal real positive root of $1-Q(z)$ with $|z|<\rho$).
	\begin{enumerate}
		\item If $\rho_0 < \rho$, the standardized random variable
			\begin{align*}
				\frac{X_n-\mu n}{\sigma \sqrt{n}}, &
				& \text{with} &&
				\mu &= \frac{1}{\rho_0 Q'(\rho_0)} 	
				& \text{and} &&
				\sigma^2 &= \frac{\rho_0 Q''(\rho_0) + Q'(\rho_0) - \rho_0 Q'(\rho_0)^2}{\rho_0^2 Q'(\rho_0)^3},
			\end{align*}	
			converges in law to a standard Gaussian variable $\Nc(0,1):$
			\begin{align*}
				\lim_{n \to \infty} \PR \left( \frac{X_n-\mu n}{\sigma \sqrt{n}} \leq x \right) &=
					\frac{1}{\sqrt{2 \pi}} \int_{- \infty}^{x} e^{-y^2/2} \, dy.
			\end{align*}
			
		\item If $\rho_0 = \rho$, the normalized random variable
				$\frac{X_n}{\vartheta \sqrt{n}}$, with $\vartheta = \frac{\sqrt{2}}{\Qcsing}$,
				converges in law to a Rayleigh distributed random variable with density $x e^{-x^2/2}$:
				\begin{align*}
					\lim_{n \to \infty} \PR \left( \frac{X_n}{\vartheta \sqrt{n}} \leq x \right) &=
						1 - e^{-x^2/2}.
				\end{align*}
			In particular, the average number of catastrophes is given by $\E(X_n) \sim \frac{1}{\Qcsing} \sqrt{\pi n}$.
			
		\item If $\rho_0$ does not exist, the limit distribution is a discrete one:
			\begin{align*}
				\PR \left( X_n = k \right) &=  \frac{\left( k \Qcsing/\lambda +  C/\tau \right) \lambda^k}{\Qcsing D(\rho)^2 + C/\tau D(\rho)}  \left( 1 + \LandauO\left( \frac{1}{n} \right) \right), 
			\end{align*}
			where $\Qcsing$ is defined as in Theorem~\ref{theo:Dasym}, $\lambda = Q(\rho)$, $C = \sqrt{2 \frac{P(\tau)}{P''(\tau)}}$, and $\tau > 0$ is the unique real positive root of $P'(\tau)=0$.
			In particular, $X_n$ converges to the random variable given 
by the following sum of two negative binomial distributions\footnote{The negative binomial distribution of parameters
$r$ and $\lambda$ is 
defined by $\PR(X=k)= \binom{k+r-1}{k} \lambda^k (1-\lambda)^r$.}:
			$$
				\Qcsing \operatorname{NegBinom}(2,\lambda) + \frac{C}{\tau} \operatorname{NegBinom}(1,\lambda).
			$$
	\end{enumerate}
\end{theorem}

\begin{proof}
	First, for $\rho_0 < \rho$ we see from~\eqref{eq:gfnrcat} that we are in the case of a perturbed supercritical composition scheme from Proposition~\ref{prop:pertsupcritseq}.
	It is supercritical because $Q(z)$ is singular at $\rho_0$ and $\lim_{z \to \rho_0} Q(z) = \infty$.
	The perturbation $E(z)$ is analytic for $|z| < \rho$, and the other conditions are also satisfied.
	Hence, we get convergence to a normal distribution. 
	
	Second, for $\rho_0 = \rho$, we start with the asymptotic expansion of $E(z)$ at $z \sim \rho$. Due to \citet*[Theorem~3]{BaFl02} we have
	\begin{align}
		\label{eq:EexpBF}
		E(z) &= E(\rho)\left(1  - \frac{C}{\tau} \sqrt{1-z/\rho} \right) + \LandauO(1-z/\rho), \qquad \text{	for $z \sim \rho$. }
	\end{align} 
This implies by~\eqref{eq:Drecexpanded} the asymptotic expansion
	\begin{align*}
		\frac{1}{C(z,v)} &= \frac{1}{E(\rho)} \left( (1-v) + \Qcsing \sqrt{1-z/\rho}\right) + \LandauO(1-z/\rho) + \LandauO\left((1-v)\sqrt{1-z/\rho}\right),
	\end{align*}
	for $z \sim \rho$ and $v \sim 1$.
	The shape above is the one necessary for the Drmota--Soria limit scheme in \citet*[Theorem~1]{DrSo97} which implies a Rayleigh distribution.
	By a variant of the implicit function theorem applied to the small roots, the function satisfies the analytic continuation properties required to apply this theorem. 
	
	Third, we know by Theorem~\ref{theo:Dasym} that $D(z)$ possesses a square-root singularity. Thus, combining the expansions~\eqref{eq:Drecexpanded}, \eqref{eq:Dexp}, and~\eqref{eq:EexpBF} we get the asymptotic expansion of $C(z,v)$, which is of the same type of a square root as the one from Theorem~\ref{theo:exc}. Extracting coefficients with the help of singularity analysis and normalizing by the result of Theorem~\ref{theo:exc} shows the claim.
\end{proof}

Let us end this discussion with an application to Dyck paths.

\begin{corollary}
	\label{coro:limitcat}
	The number of catastrophes of a random Dyck path with catastrophes of length $n$ is normally distributed. Let $\mu$ be the unique real positive root of $31 \mu^3 + 31 \mu^2 + 40 \mu - 3$, and $\sigma$ be the unique real positive root of $29791 \sigma^6-59582 \sigma^4+60579 \sigma^2-2927$. The standardized version of $X_n$,
	\begin{align*}
		\frac{X_n-\mu n}{\sigma \sqrt{n}}, &
		& \text{with} &&
		\mu &\approx 0.0708358118 
		& \text{and} &&
		\sigma^2 &\approx 0.05078979113,
	\end{align*}	
	converges in law to a Gaussian variable $\Nc(0,1)$.
\end{corollary}

\subsection{Average number of returns to zero}

In order to count the number of returns to zero, we decompose $F_0(z)$ into a sequence of arches. Let $A(z)$ be the corresponding generating function. 
(Caveat: this is not the same generating function as $A(z)$ in Proposition~\ref{prop:arches}.) Then,
\begin{align*}
	A(z) &= 1 - \frac{1}{F_0(z)}. 
\end{align*}
Let $g_{n,k}$ be the number of excursions with catastrophes of length $n$ and $k$ returns to zero. Then,
\begin{align*}
	G(z,v) := \sum_{n,k \geq 0} g_{n,k} z^n v^k = \frac{1}{1-vA(z)}.
\end{align*}

From now on, let $X_n$ be the random variable giving the number of returns to zero in excursions with catastrophes of length $n$ drawn uniformly at random:
\begin{align*}
	\PR\left(X_n = k\right) &= \frac{[z^nv^k]G(z,v)}{[z^n]G(z,1)}.
\end{align*}

	Applying the same ideas and techniques we used in the proof of Theorem~\ref{theo:limitcat2}, we get the following result.
\pagebreak
\begin{theorem}[Limit law for the number of returns to zero]
	\label{theo:limitretzero}
	The number of returns to zero of a random excursion with catastrophes of length $n$ admits a limit distribution, with the limit law being dictated by the relation between the singularities $\rho_0$ and $\rho$.
	\begin{enumerate}
		\item If $\rho_0 < \rho$, the standardized random variable
			\begin{align*}
				\frac{X_n-\mu n}{\sigma \sqrt{n}}, &
				& \text{with} &&
				\mu &= \frac{1}{\rho_0 A'(\rho_0)} 	
				& \text{and} &&
				\sigma^2 &= \frac{\rho_0 A''(\rho_0) + A'(\rho_0) - \rho_0 A'(\rho_0)^2}{\rho_0^2 A'(\rho_0)^3},
			\end{align*}	
			converges in law to a standard Gaussian variable $\Nc(0,1)$.
			
		\item If $\rho_0 = \rho$, the normalized random variable
			\begin{align*}
				\frac{X_n}{\vartheta \sqrt{n}}, 
				\qquad \text{with} \qquad
				\vartheta = \sqrt{2}\frac{E(\rho)}{\Qcsing},
			\end{align*}
			converges in law to a Rayleigh distributed random variable with density $x e^{-x^2/2}$. 
			In particular, the average number of returns to zero is given by $\E(X_n) \sim \frac{E(\rho)}{\Qcsing} \sqrt{\pi n}$.
			
		\item If $\rho_0$ does not exist, the limit distribution is $\operatorname{NegBinom}(2,\lambda)$:
			\begin{align*}
				\PR \left( X_n = k \right) &= \frac{n \lambda^n}{F_0(\rho)^2}  \left( 1 + \LandauO\left( \frac{1}{n} \right) \right), & \text{with} && 
				\lambda &= A(\rho) = 1-\frac{1}{F_0(\rho)}.
			\end{align*}
	\end{enumerate}
\end{theorem}

Again, we give the concrete statement for Dyck paths with catastrophes.

\begin{corollary}
	\label{coro:dyckretzero}
	The number of returns to zero of a random Dyck path with catastrophes of length~$n$ is normally distributed. 
	Let $\mu$ be the unique real positive root of 
	$
	31 \mu^3 - 62 \mu^2 + 35 \mu - 3
	$, and $\sigma$ be the unique real positive root of 
	$
	29791 \sigma^6 + 231 \sigma^2-79
	$.	
	The standardized version of $X_n$,
	\begin{align*}
		\frac{X_n-\mu n}{\sigma \sqrt{n}}, &
		& \text{with} &&
		\mu &\approx 0.1038149281 
		& \text{and} &&
		\sigma^2 &\approx  0.1198688826,
	\end{align*}	
	converges in law to a Gaussian variable $\Nc(0,1)$.
\end{corollary}

It is interesting to compare the results of Corollaries~\ref{coro:limitcat} and \ref{coro:dyckretzero} for Dyck paths: more than $10\%$ of all steps are returns to zero, and more than $7\%$ are catastrophes. This implies that among all returns to zero approximately $70\%$ are catastrophes and $30\%$ are $-1$-jumps. Note that the expected number of returns to zero of classical Dyck paths converges to the constant~$3$.

\pagebreak
\subsection{Average final altitude}\label{sec:Finalt}
In this section we want to analyse the final altitude of a path after a certain number of steps. The \emph{final altitude} of a path is defined as the ordinate of its endpoint. Theorem~\ref{theo:LukaCatGF} already encodes this parameter using $u$:
\begin{align*}
	F(z,u) &= D(z) M(z,u), &
	M(z,u) &= \frac{\prod_{i=1}^c(u-u_i(z))}{u^c(1-zP(u))},
\end{align*}
where $M(z,u)$ is the bivariate generating function of meanders.

Let $X_n$ be the random variable giving the final altitude of paths with catastrophes of length $n$ drawn uniformly at random:
\begin{align*}
	\PR\left(X_n = k\right) &= \frac{[z^n u^k]F(z,u)}{[z^n]F(z,1)}.
\end{align*}

This random variable exhibits an interesting periodic
behaviour, as can be observed in Figure~\ref{fig:fractalluca},
and is more formally stated in the following theorem.

\begin{theorem}[Limit law for the final altitude]
	\label{theo:finalt}
	The final altitude of a random lattice path with catastrophes of length $n$ admits a discrete limit distribution:
	\begin{align}\label{omega}
		\lim_{n \to \infty} \PR\left(X_n = k\right) &= [u^k] \, \omega(u), & \text{ where }
			\omega(u) &= 
			\begin{cases}
				 \prod_{\ell = 1}^d \frac{1-v_{\ell}(\rho_0)}{u-v_{\ell}(\rho_0)} & \text{if } \rho_0 \leq \rho, \\[-1mm] \\[-1mm]
				 \frac{\Qcsing D(\rho) + \frac{C}{\tau-u}}{\Qcsing D(\rho) + \frac{C}{\tau-1}} \prod_{\ell = 1}^d \frac{1-v_{\ell}(\rho)}{u-v_{\ell}(\rho)}  & \text{if $\rho_0$ does not exist}. 				
			\end{cases}
	\end{align}
\end{theorem}
\begin{proof}
	Let us distinguish three cases. First, in the case of $\rho_0<\rho$ the function $D(z)$ is responsible for the singularity of $F(z,u)$. Thus, by~\citet*[Problem~$178$]{polyaszego25} (see also \citet*[Theorem~VI.12]{flaj09}) we get the asymptotic expansion
	\begin{align*}
		\lim_{n \to \infty} \frac{[z^n] F(z,u)}{[z^n]F(z,1)} &= \frac{M(\rho_0,u)}{M(\rho_0,1)} = \prod_{\ell = 1}^d \frac{1-v_{\ell}(\rho_0)}{u-v_{\ell}(\rho_0)}.
	\end{align*}
	
	For the other cases by Lemma~\ref{lem:Ddenom} we require $\delta <0$. Then we know from~\citet*[Theorem~6]{BaFl02} that $M(z,u)$ admits a discrete limit distribution. The function $M(z,u)$ admits the expansion
	\begin{align*}
		M(z,u) &= M(\rho,u) \left( 1 + \frac{C}{u-\tau} \sqrt{1-z/\rho}\right) + \LandauO\left(1-z/\rho\right), \text{\qquad	for $z \to \rho$. }
	\end{align*}

	In the second case ($\rho_0=\rho$) and third case (if $\rho_0$ does not exist) we derive the expansion of $F(z,u)$ by multiplying this expansion with the one of $D(z)$ from~\eqref{eq:Drho0expanded} and~\eqref{eq:Dexp}, respectively. Normalizing with the results of Theorem~\ref{theo:mea} yields the result.
\end{proof}

\begin{corollary}
	\label{coro:finaltdyck}
	The final altitude of a random Dyck path with catastrophes of length $n$ admits a geometric limit distribution with parameter $\lambda = v_1(\rho_0)^{-1} \approx 0.6823278$:
	\begin{align*}
		\PR\left(X_n = k\right) &\sim \left(1 - \lambda\right) \lambda^k.
	\end{align*}	
	The parameter is the unique real positive root of $\lambda^3+\lambda-1$ and is given by 
	$$\lambda = \frac{1}{6} \left(108 + 12 \sqrt{93}\right)^{1/3} - 2\left(108 + 12 \sqrt{93}\right)^{-1/3}.$$
\end{corollary}

The nature of this result changes to a periodic one for different step polynomials as seen in Figure~\ref{fig:fractalluca}.

\begin{figure}[h!b]
	\begin{center}	
		\includegraphics[width=0.29\textwidth,trim={0 0 0 34mm},clip]{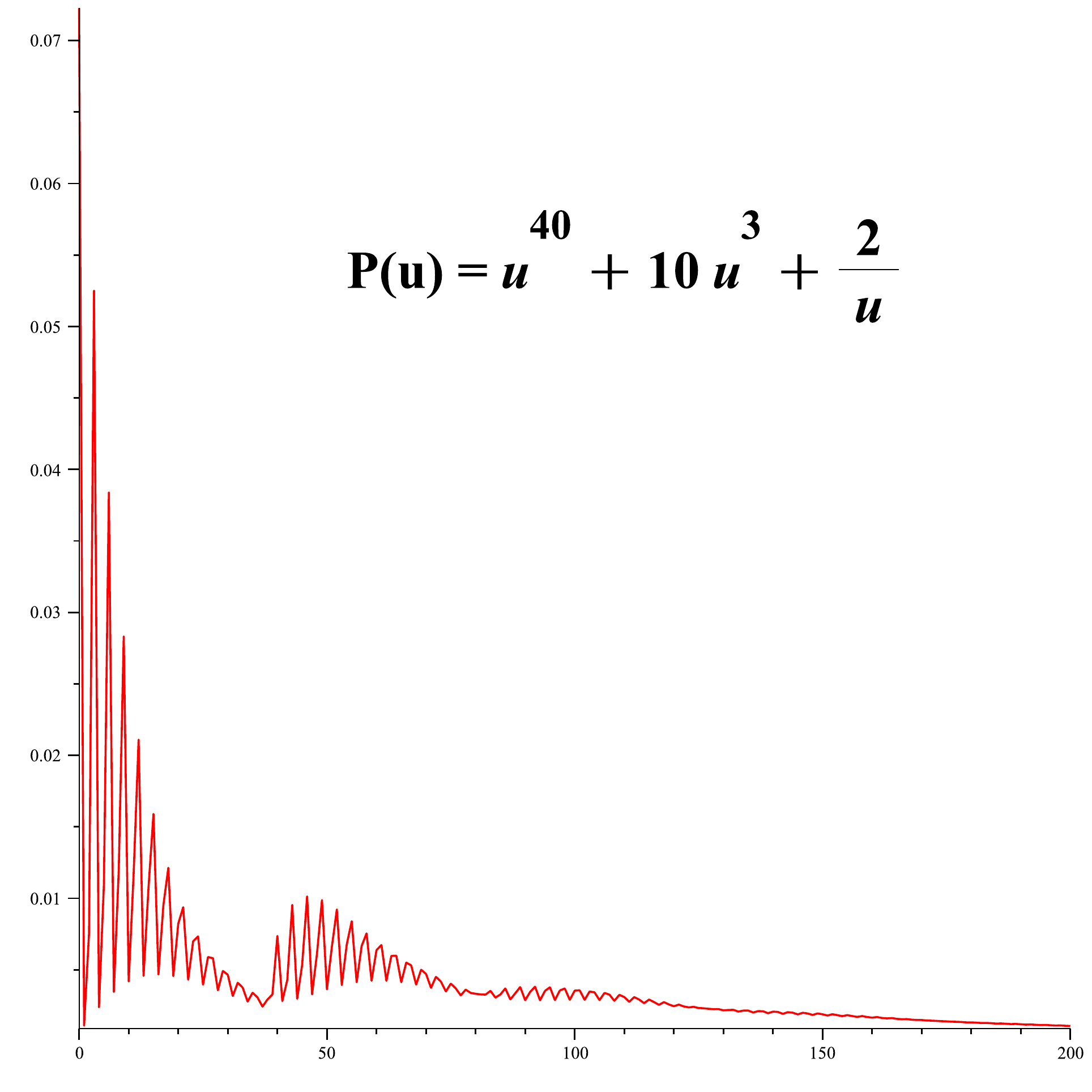}
		\includegraphics[width=0.29\textwidth,trim={0 0 0 34mm  },clip]{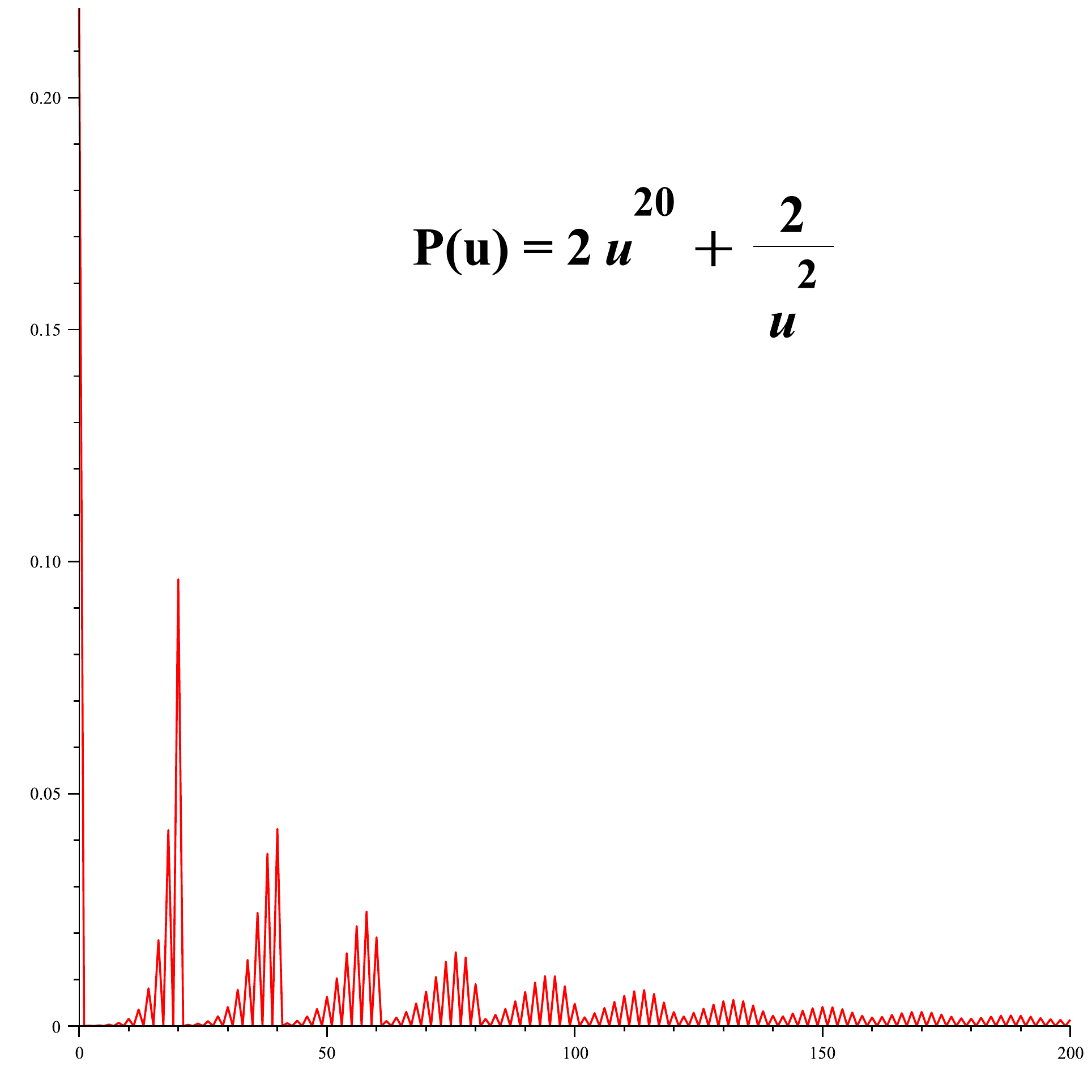}
		\includegraphics[width=0.29\textwidth,trim={0 0 0 34mm },clip]{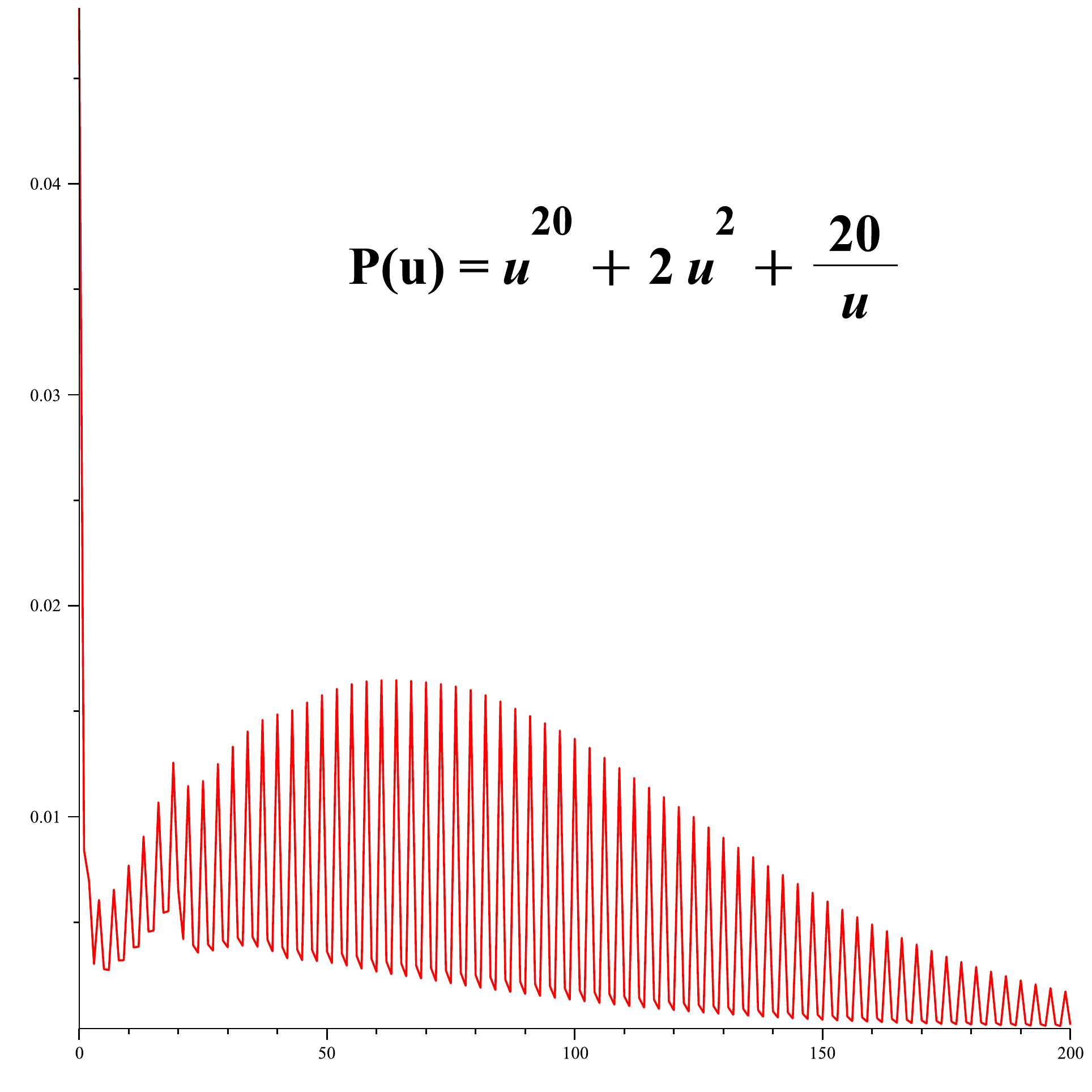}\\
		\includegraphics[width=0.29\textwidth,trim={0 0 0 26mm },clip]{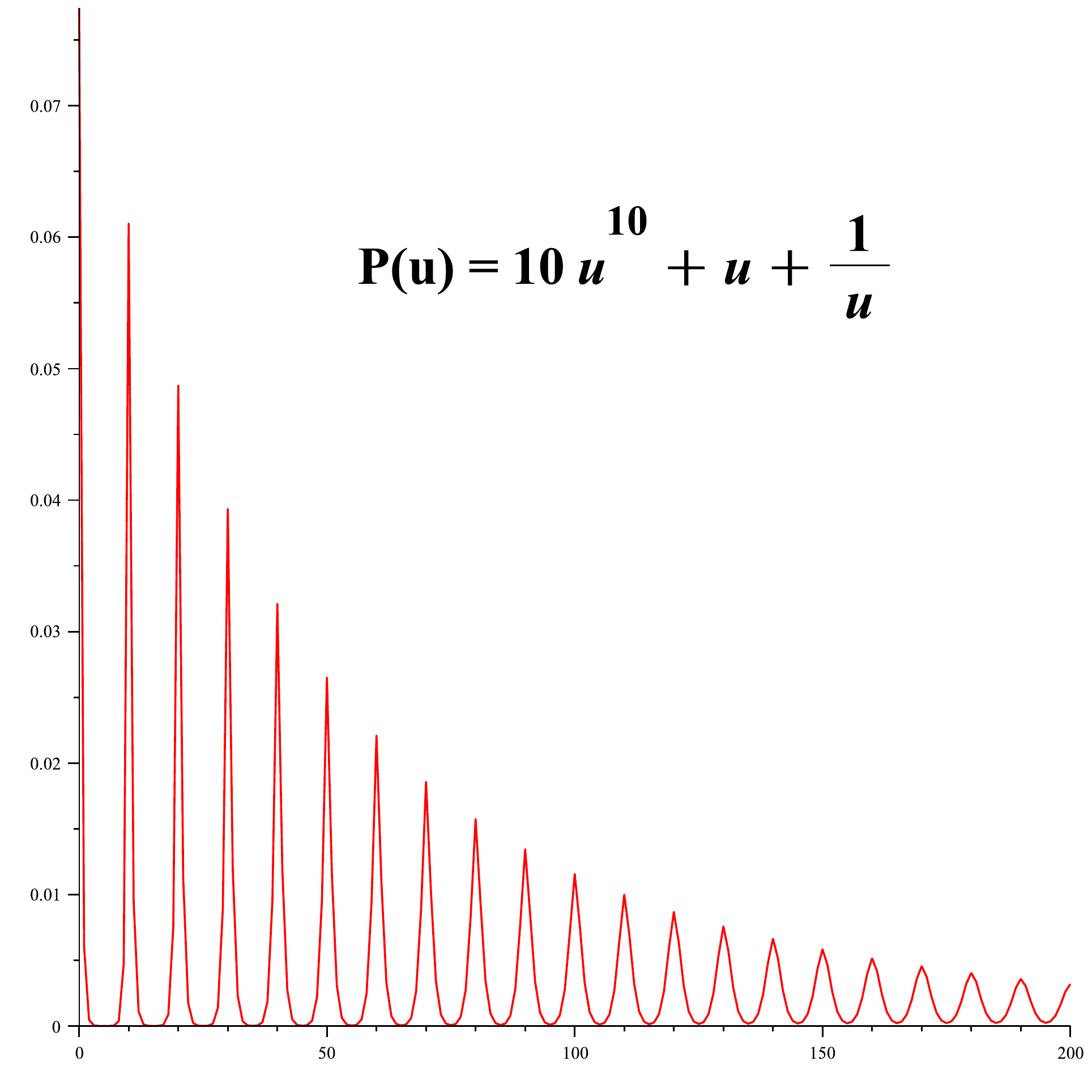}
		\includegraphics[width=0.29\textwidth,trim={0 0 0 26mm },clip]{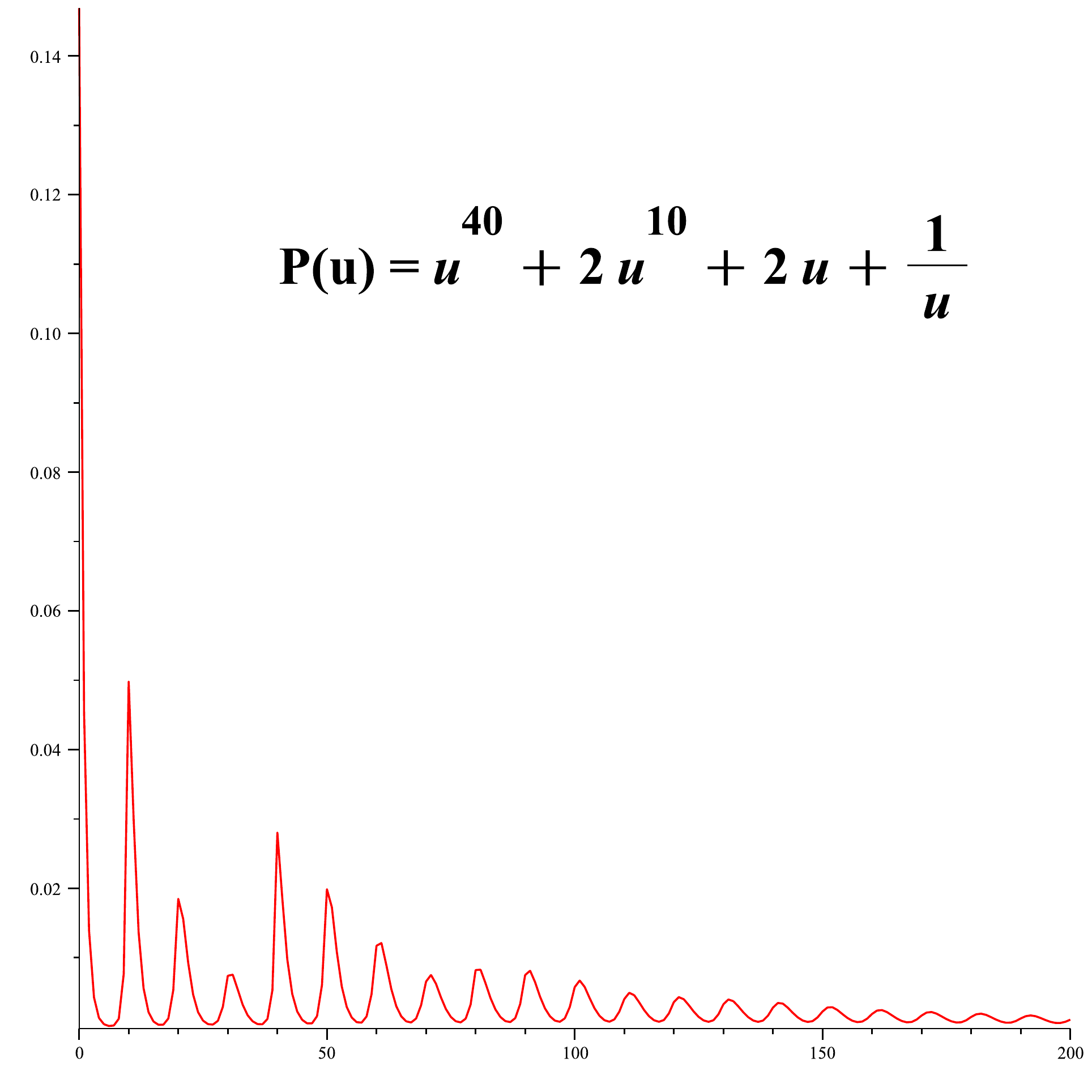}
		\includegraphics[width=0.29\textwidth,trim={0 0 0 26mm },clip]{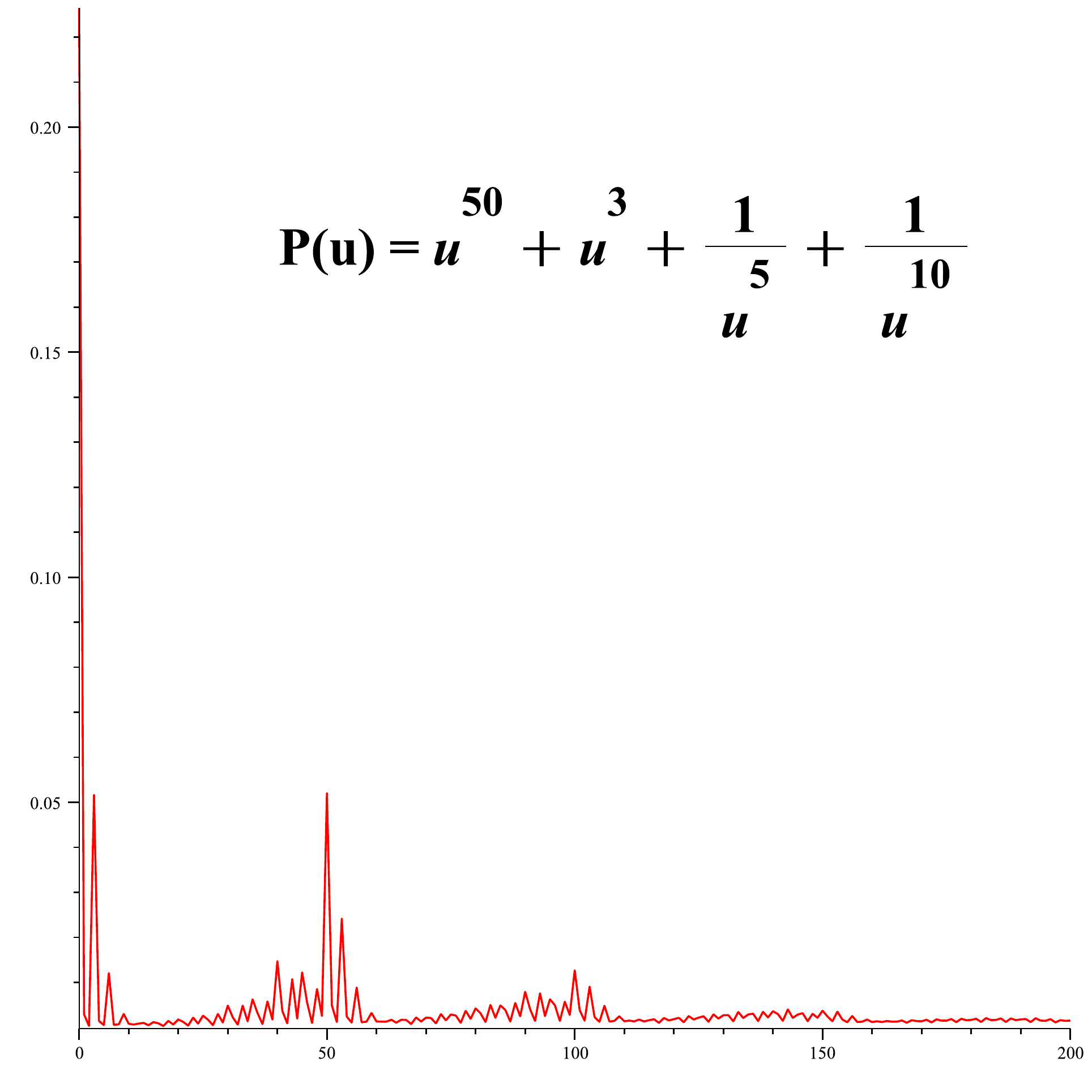}\\
		\caption{The final altitude follows a discrete limit law, which depends on the jump polynomial $P(u)$. 
The $x$-axis is labeled with $k$ and the $y$-axis gives $\PR(\text{final altitude}=k)$.}
		\label{fig:fractalluca}
	\end{center}
\end{figure}

The main periodicity observed in these pictures is due to the fact 
that the limit law is a sum of ``geometric limit laws'' of {\em complex} parameter (as given in Theorem~\ref{theo:finalt}). 
The pictures show a combination of a ``macroscopic'' and a ``microscopic'' behaviour.
On the macroscopic level we see a period of the size of the largest positive jump $d$.
On the microscopic level we see smaller fluctuations related to the small jumps (with some additional periodic behaviour if the support of these small jumps is periodic).

We observe that for some values of $k$, $\PR(X_n=k)$ is very close to zero, while it is not the case for nearby values of $k$.
This noteworthy phenomenon has links with the Skolem--Pisot problem (i.e., deciding if a rational function $R(u)\in \Z[[u]]$ has 
a zero term in its Taylor expansion, see e.g.~\citet*{OuaknineWorrell14} for recent progress).
In fact, a partial fraction expansion of $\omega(u)$ from Equation~\eqref{omega} gives a closed-form expression for $\PR(\text{final altitude}=k)$ in terms of powers of the poles of $\omega(u)$, 
 which  dictate how close to zero our limit laws can get.

\pagebreak

\subsection{Cumulative size of catastrophes}\label{TotalAmplitude}

Another interesting parameter is the \emph{cumulative size of catastrophes} of excursions of length $n$. 
Thereby we understand the sum of sizes of all catastrophes contained in the path.
Let $a_{n,k}$ be the number of excursions with catastrophes of length~$n$ and cumulative size of catastrophes $k$. Then its bivariate generating function $A_{\text{cum}}(z,u) = \sum_{n,k \geq 0} a_{n,k} z^n u^k$ is given by
\begin{align*}
	A_{\text{cum}}(z,u) &= D(z,u) E(z), \qquad \text{where} \\
	D(z,u) &= \frac{1}{1-Q(z,u)} \qquad \text{~~~and} \\
	Q(z,u) &= zq \left( M(z,u) - E(z) - \sum_{-j \in \J_-} u^j M_{j}(z) \right).
\end{align*}
The generating function $Q(z,u)$ keeps track of the sizes of used catastrophes. The new parameter~$u$ does not influence the singular expansion of $Q(z)$ analysed in Theorem~\ref{theo:Dasym}. We get for $z \to \rho-$ and $0\leq u \leq1$ the expansion
\begin{align}
	\label{eq:Qzusing}
	Q(z,u) &= Q(\rho,u) - \eta(u) \sqrt{1-z/\rho} + \LandauO(1-z/\rho),
\end{align}
where $\eta(u)$ is a non-zero function, and in terms of the previous expansion of $Q(z)$ we have $\eta(1) = \eta$. 

Let $X_n$ be the random variable giving the cumulative size of catastrophes in lattice paths with catastrophes of length $n$ drawn uniformly at random:
\begin{align*}
	\PR\left(X_n = k\right) &= \frac{[z^n u^k]A_{\text{cum}}(z,u)}{[z^n]A_{\text{cum}}(z,1)}.
\end{align*}

\begin{theorem}[Limit law for the cumulative size of catastrophes]
	\label{theo:limittotamp}
	The cumulative size of catastrophes of a random excursion with catastrophes of length $n$ admits a limit distribution, with the limit law being dictated by the relation between the singularities $\rho_0$ and $\rho$.
	\begin{enumerate}
		\item If $\rho_0 < \rho$, the standardized random variable
			\begin{equation*}
				\frac{X_n-\mu n}{\sigma \sqrt{n}},  \qquad  \text{with} \qquad
				\mu = \frac{Q_u(\rho_0,1)}{\rho_0 Q_z(\rho_0,1)} \qquad \text{and}	
			\end{equation*}	
			\begin{equation*}
				\sigma^2 = \left(1+\frac{\rho_0 Q_{uu}(\rho_0,1)}{Q_z(\rho_0,1)}\right) \mu^2 +
				            \left(1-\frac{2 Q_{zu}(\rho_0,1)}{Q_z(\rho_0,1)}
				                   +\frac{Q_{zz}(\rho_0,1)}{Q_z(\rho_0,1)}\right) \mu,
			\end{equation*}	
			converges for $\sigma^2>0$ in law to a standard Gaussian variable $\Nc(0,1)$.
			
		\item If $\rho_0 = \rho$, the normalized random variable
			\begin{align*}
				\frac{X_n}{\vartheta \sqrt{n}}, 
				\qquad \text{with} \qquad 
				\vartheta = \sqrt{2} \frac{Q_u(\rho,1)}{\Qcsing},
			\end{align*}
			converges in law to a Rayleigh distributed random variable with density $x e^{-x^2/2}$.
			In particular, the average cumulative size of catastrophes is here $E[X_n]= \frac{Q_u(\rho,1)}{\Qcsing} \sqrt{\pi n}$.
			
		\item If $\rho_0$ does not exist, the limit distribution is discrete and given by:
			\begin{align*}
				\lim_{n \to \infty} \PR \left( X_n = k \right) &= [u^k] \frac{ \eta(u) D(\rho,u)^2 + \frac{C}{\tau} D(\rho,u)}{\eta D(\rho)^2 + \frac{C}{\tau} D(\rho)}.
			\end{align*}
	\end{enumerate}
\end{theorem}

\begin{proof}
	In the first case $\rho_0 < \rho$ we will use the meromorphic scheme from~\citet*[Theorem~IX.9]{flaj09}, which is a generalization of Hwang's quasi-power theorem. 
	In order to apply it we need to check three conditions. First, the \emph{meromorphic perturbation condition}: 
	We know already from the proof of Theorem~\ref{theo:Dasym} that $\rho_0$ is a simple pole. 
	What remains is to show that in a domain $\Dc = \{ (z,u)~:~|z|<r,|u-1|<\varepsilon\}$ the function admits the following representation
	\begin{align*}
		A_{\text{cum}}(z,u) &= \frac{B(z,u)}{C(z,u)},
	\end{align*}
	where $B(z,u)$ and $C(z,u)$ are analytic for $(z,u) \in \Dc$. There exists a $\delta>0$ such that $r := \rho_0 + \delta < \rho$. For this value the representation holds, as $B(z,u) = u^c(1-zP(u))E(z)$ and $C(z,u) = u^c(1-zP(u)) - zq\prod_{i=1}^c (1-u_i(z))$ are only singular for $z=\rho$ or $u=0$.

	Next, the \emph{non-degeneracy} $Q_u(\rho,1) Q_z(\rho,1) \neq 0$ is easily checked. It ensures the existence of a non-constant $\rho(u)$ analytic at $u=1$, such that  $1-Q(\rho(u),u)=0$.
	
	Finally, the \emph{variability condition} $r''(1) + r'(1) -r'(1)^2 \neq 0$ for $r(u) = \frac{\rho(1)}{\rho(u)}$ is also satisfied due to 
	\begin{align*}
		\rho(1) &= \rho_0, \qquad \rho'(1)  = - \frac{Q_u(\rho,1)}{Q_z(\rho,1)},\\
		\rho''(1) &= -\frac{1}{Q_z(\rho,1)} \left( Q_{zz}(\rho,1) \rho'(1) + 2 Q_{z,u}(\rho,1) \rho'(1) + Q_{uu}(\rho,1) \right).
	\end{align*}	
	This implies the claimed normal distribution.
	
	In the second case $\rho_0 = \rho$ we apply again the Drmota--Soria limit theorem~\citet*[Theorem~1]{DrSo97} which leads to a Rayleigh distribution. 
	As $Q(\rho_0,1) = 1$, like in~\eqref{eq:Qleadvanish}, we have a cancellation of the constant term in the Puiseux expansion (for $z \sim \rho$ and $u \sim 1$). Thus, using the asymptotic expansions~\eqref{eq:EexpBF} and \eqref{eq:Qzusing} leads to
	\begin{align*}
		\frac{1}{A_{\text{cum}}(z,u)} &= \frac{Q_u(\rho,1)}{E(\rho)}(1-u) + \frac{\eta}{E(\rho)} \sqrt{1-z/\rho} \, + \\
		& \qquad \LandauO\left(1-z/\rho\right) + \LandauO\left((u-1)(1-z/\rho)\right) + \LandauO\left((u-1)^2\right).
	\end{align*}
	Note that the analyticity and the other technical conditions required to apply this theorem follow from the respective properties of the generating functions $M(z,u), E(z),$ and $M_j(z)$. This implies the claimed Rayleigh distribution with the normalizing constant $\vartheta = \sqrt{2} \frac{Q_u(\rho,1)}{\Qcsing}$.
	
	In the third case, if $\rho_0$ does not exist, the singularity arises at $z=\rho$. In particular, there arises no zero in the denominator. Thus, after combining the known expansions~\eqref{eq:EexpBF} and \eqref{eq:Qzusing}, singularity analysis yields the given discrete form. This implies the claimed discrete distribution.
\end{proof}

\begin{corollary}
	\label{coro:dycktotalcat}
	The cumulative size of catastrophes of a random Dyck path with catastrophes of length~$n$ is normally distributed. 		
	Let $\mu$ be the unique real positive root of 
	$
	31\mu^3+62\mu^2+71\mu-27
	$, and $\sigma$ be the unique real positive root of 
	$
	29791\sigma^6-59582\sigma^4+298411\sigma^2-159099
	$.
	The standardized version of $X_n$,
	\begin{align*}
		\frac{X_n-\mu n}{\sigma \sqrt{n}}, &
		& \text{with} &&
		\mu &\approx 0.2938197987 
		& \text{and} &&
		\sigma^2 &\approx  0.5809693987,
	\end{align*}
	converges in law to a Gaussian variable $\Nc(0,1)$.
\end{corollary}

\subsection{Size of an average catastrophe}\label{AverageAmplitude}

As one of the last parameters of our lattice paths with catastrophes, we want to determine the law behind the size of a random catastrophe among all lattice paths of length~$n$. In other words, one draws uniformly at random a catastrophe among all possible catastrophes of all lattice paths of length~$n$. Note that this is also the law behind the size of the first (or last) catastrophe, as cyclic shifts of excursions ending with a catastrophe transform any catastrophe into the first (or last) one.

We can construct it from the generating function counting the number of catastrophes. It is given in~\eqref{eq:gfnrcat} where each catastrophe is marked by a variable $v$. 

\begin{lemma}
The bivariate generating function $A_{\text{avg}}(z,u)$ marking the size of a random catastrophe among all excursions with catastrophes is given by
\begin{align*}
	A_{\text{avg}}(z,u) &= E(z) + Q(z,u) D(z)^2 E(z).
\end{align*}
\end{lemma}
\begin{proof}
A random excursion with catastrophes either contains no catastrophes and is counted by $E(z)$, or it contains at least one catastrophe. In the latter we choose one of its catastrophes and its associated excursion ending with this catastrophe. Then we replace it with an excursion ending with a catastrophe whose size has been marked. 
This corresponds to
\begin{align*}
	A_{\text{avg}}(z,u) &= E(z) + \frac{Q(z,u)}{Q(z)} \left. \frac{\partial}{\partial v} C(z,v) \right|_{v=1}.
\end{align*}
Computing this expression proves the claim.
\end{proof}

As before we define a random variable $X_n$ for our parameter as
\begin{align*}
	\PR\left(X_n = k\right) &= \frac{[z^n u^k]A_{\text{avg}}(z,u)}{[z^n]A_{\text{avg}}(z,1)}.
\end{align*}

Due to the factor $Q(z,u)$ the situation is similar to final altitude in Section~\ref{sec:Finalt}. 

\begin{theorem}[Limit law for the size of a random catastrophe]
	The size of a random catastrophe of a lattice path of length $n$ admits a discrete limit distribution:
	\begin{align*}
		\lim_{n \to \infty} \PR\left(X_n = k\right) &= [u^k] \, \omega(u), \qquad \text{ where } \\
			\omega(u) &= 
			\begin{cases}
				 Q(\rho_0,u) & \text{ if } \rho_0 \leq \rho, \\[1mm]
				 \frac{ \frac{C}{\tau} + \left(\frac{C}{\tau} D(\rho)^2 + 2 \Qcsing D(\rho)^3 \right) Q(\rho,u) +  \Qcsing(u) D(\rho)^2 }{\frac{C}{\tau} + \left(\frac{C}{\tau} D(\rho)^2 + 2 \Qcsing D(\rho)^3 \right) Q(\rho,1) +  \Qcsing D(\rho)^2}  \text{(sic!)} & \text{ if $\rho_0$ does not exist}. 				
			\end{cases}
	\end{align*}
\end{theorem}

\begin{proof}
	The proof is similar to the one of Theorem~\ref{theo:finalt}. 
	First, for $\rho_0<\rho$ it holds that only $D(z)^2$ is singular at $\rho_0$, where all other terms are analytic. Thus, by \citet*[Problem~$178$]{polyaszego25} the claim holds. It is then possible to extract a closed-form expression for the coefficients  of an algebraic function, via 
the Flajolet--Soria formula which is discussed in~\cite{BanderierDrmota15}, 
but this expression of  $[u^k]	Q(\rho_0,u)$ in terms of nested sums of binomials is here too big to be useful. 

	Second, in the case $\rho_0=\rho$ we combine the singular expansions~\eqref{eq:Drho0expanded}, \eqref{eq:EexpBF}, and \eqref{eq:Qzusing} to get
	\begin{align*}
		A_{\text{avg}}(z,u) &= \frac{E(\rho) Q(\rho,u)}{\eta^2 (1-z/\rho)} + \LandauO\left((1-z/\rho)^{-1/2}\right). 
	\end{align*}
	In other words, the polar singularity of $D(z)^2$ dominates, and the situation is similar to the one before. 
	
	In the final case when $\rho_0$ does not exist, we again combine the singular expansions. This time the expansion of $D(z)$ is given by~\eqref{eq:Dexp}. This implies a contribution of all terms, as all of them are singular at once and all of them have the same type of singularity. 
\end{proof}

\begin{corollary}
	Let $\lambda$ be the unique real positive root of $\lambda^3+\lambda-1$ and given by 
	$$\lambda = \frac{1}{6} \left(108 + 12 \sqrt{93}\right)^{1/3} - 2\left(108 + 12 \sqrt{93}\right)^{-1/3}.$$
	The size of a random catastrophe among all Dyck paths with catastrophes of length $n$ admits a 
	(shifted) geometric limit distribution with parameter $\lambda \approx 0.6823278$:
	\begin{align*}
		\lim_{n \to \infty} \PR\left(X_n = k\right) &= 
			\begin{cases}
				\left(1 - \lambda\right) \lambda^{k-2}, & \text{ for } k \geq 2, \\
				0, & \text{ for } k=0,1. 
			\end{cases}
	\end{align*}
\end{corollary}

Comparing this result to the one for the final altitude of meanders in Corollary~\ref{coro:finaltdyck}, we see that the type of the law is of the same nature (yet shifted for the size of catastrophes), and that the parameter~$\lambda$ is the same.  The following lemma explains this connection.

\begin{lemma}
	\label{lem:connectionQandMd1}
	Let $d=1$, i.e., $P(u) = p_{-c} u^{-c} + \cdots + p_1 u^1$ be the jump polynomial. Then, the generating function of excursions of length $n$ (marked by $z$) ending with a catastrophe of size $k$ (marked by $u$) admits the decomposition
	\begin{align*}
		Q(z,u) &=  q p_{1} z^2 u^{c+1} M(z,u) M_{c}(z) + qz \sum_{\substack{-j \notin \J_- \\ -c < j < 0 }} u^{-j} M_{j}(z).
	\end{align*}
\end{lemma}

\begin{proof}
	The idea is a last passage decomposition with respect to reaching level $c+1$. First, assume that $p_{-c}, \ldots,p_{-1} \neq 0$. Then the smallest catastrophe is of size $c+1$. 
	We decompose the excursion with respect to the last jump from altitude $c$ to altitude $c+1$, see Figure~\ref{fig:Qdecompd1}. Left of it, there is a meander ending at altitude $c$, and right of it there is a meander starting at altitude $c+1$ and always staying above altitude $c+1$. The size of the ending catastrophe is then given by the final altitude plus $c+1$, this gives the factor $z u^{c+1} M(z,u)$.
	This proves the first part. 
	
	For the second part, note that if one of the $p_{-c},\ldots,p_{-1}$ is equal to $0$, then catastrophes of the respective size are allowed. These are given by meanders ending at this altitude and a jump to~$0$. 
\end{proof}\pagebreak

\begin{figure}[ht]
	\begin{center}	
		\includegraphics[width=0.8\textwidth]{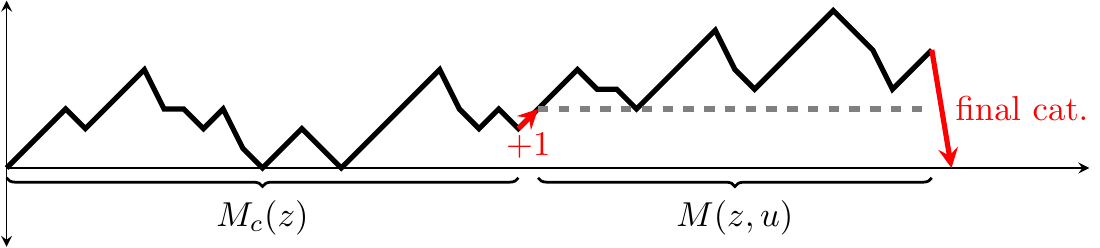}
		\caption{Decomposition of an excursion ending with a catastrophe for $d=1$ from Lemma~\ref{lem:connectionQandMd1}.}
		\label{fig:Qdecompd1}
	\end{center}
\end{figure}

This lemma shows that the probability generating function of the final altitude and of the size of an average catastrophe are connected. In particular, for $d=1$ and $\rho_0 < \rho$ we have
\begin{align*}
	\lim_{n \to \infty} \frac{[z^n]M(z,u)}{[z^n]M(z,1)} &= 
		\frac{1}{u^{c+1}}
		\frac{Q(\rho_0,u) - \sum\limits_{-j \notin \J_-,-c < j < 0 } q \rho_0 u^{-j} M_{j}(\rho_0)}
		     {1 - \sum\limits_{-j \notin \J_-, -c < j < 0 } q \rho_0 M_{j}(\rho_0) }.
\end{align*}
We see the shift by $u^{-c-1}$ of the probability generating function. It is now obvious how these laws are related: the parameters are the same, there is just a shift in the parameter and a subtraction of certain, initial values. 

The results of Lemma~\ref{lem:connectionQandMd1} can be generalized to $d \geq 2$, but the explicit results are more complicated. For example, for $d=2$ there are $4$ different cases in the last passage decomposition: a $+1$-jump from $c$ to $c+1$, a $+2$-jump from $c$ to $c+2$, a $+2$-jump from altitude $c-1$ to $c+1$, all followed by a meander, and a $+2$-jump from $c$ to $c+2$ followed by a path always staying above $c+1$. 

However, in all cases there is a factor $M(z,u)$ in $Q(z,u)$ if $p_{-c},\ldots,p_{-1} \neq 0$.

\subsection{Waiting time for the first catastrophe}\label{FirstCatastrophe}

We end the discussion on limit laws with a parameter that might be of the biggest interest in applications: the waiting time for the first catastrophe. 
Let $w_{n,k}$ be the number of excursions with catastrophes of length~$n$ 
such that the first catastrophe appears at the $k$-th steps for $k >0$. Let $w_{n,0}$ be the number of such paths without a catastrophe. Then its bivariate generating function $W(z,u) = \sum_{n,k \geq 0} w_{n,k} z^n u^k$ is given by
\begin{align*}
	W(z,u) &= E(z) + Q(zu) D(z) E(z). 
\end{align*}
This is easily derived from Theorem~\ref{theo:LukaCatGF} as the prefix $D(z)$ is a sequence of excursions with only one catastrophe at the very end. Thus, marking the length of the first of such excursions marks the position of the first catastrophe.

As done repeatedly we define a random variable $X_n$ for our parameter as
\begin{align*}
	\PR\left(X_n = k\right) &= \frac{[z^n u^k]W(z,u)}{[z^n]W(z,1)}.
\end{align*}

\begin{theorem}[Waiting time of the first catastrophe]
	The waiting time for the first catastrophe in a lattice path with catastrophes of length $n$ admits a discrete limit distribution:
	\begin{align*}
		\lim_{n \to \infty} \PR\left(X_n = k\right) &= [u^k] \, \omega(u), \qquad \text{ where } &
			\omega(u) &= 
			\begin{cases}
				 Q(\rho_0 u) & \text{ if } \rho_0 \leq \rho, \\
				 1 - Q(\rho) + Q(\rho u)  & \text{ if $\rho_0$ does not exist}. 				
			\end{cases}
	\end{align*}
\end{theorem}

\begin{proof}
	The proof follows again the same lines as the one of Theorem~\ref{theo:finalt}. In this particular case, we combine the asymptotic expansion of $D(z)E(z)$ from Theorem~\ref{theo:exc} with the asymptotic expansion of $E(z)$ from \citet*[Theorem~3]{BaFl02}.
\end{proof}

In the case of Dyck paths we have
\begin{align*}
	Q(z) = \frac{1}{2z^2} \left( z^2-z-1 - (z^2+z-1)\sqrt{\frac{1+2z}{1-2z}}\right).
\end{align*}
The corresponding limit law, which consists of the sum of two discrete distributions for the odd and even waiting times,
is shown in Figure~\ref{fig:waitingtime}. We see a periodic behaviour with a distribution for the even and odd steps. 
This arises from the fact that catastrophes are not allowed at altitudes~$0$ \emph{and} $1$. 
Starting from the origin this effects only the odd numbered steps. 
The probabilities for catastrophes at an odd step are lower than the ones at the following even step, 
because we can reach altitude $1$ from below and from above, 
whereas the only restriction for the even steps is at altitude~$0$ which can only be reached from above. 
It was also interesting (and a priori not expected) 
to discover that the occurrence of the first catastrophe has a higher probability at step~$6$ than at step~$4$. 

\begin{figure}[ht]
	\begin{center}	
		\includegraphics[width=0.4\textwidth]{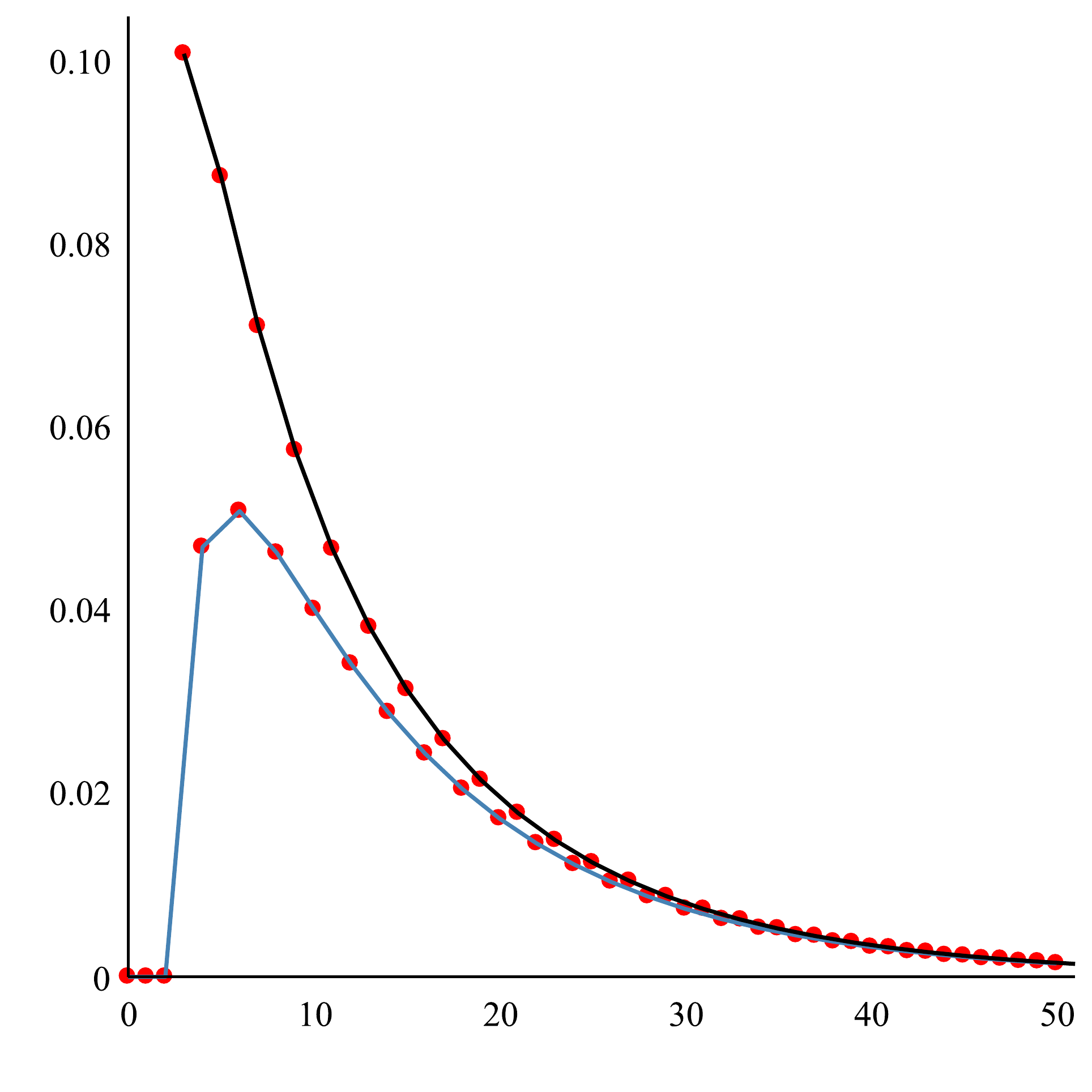}
		\caption{The red dots represent the discrete limit distribution of the waiting time for the first catastrophe in the case of Dyck paths. It is the sum of two discrete distributions: 
one for the odd waiting times (the top curve in black) and one for the even waiting times (the bottom curve in blue).}
		\label{fig:waitingtime}
	\end{center}
\end{figure}

\section{Uniform random generation}
\label{sec:catuniformgeneration}

In order to generate our lattice paths with catastrophes, 
it is for sure possible to use a dynamic programming approach;
this would require $\LandauO(n^3)$ bits in memory.
Via some key methods from the last twenty years, our next theorem shows that it is possible to do much better.
\begin{theorem}[Uniform random generation]\qquad\newline\vspace{-4mm}
\begin{compactitem}
\item {\em Dyck paths} with catastrophes can be generated uniformly at random in linear time.
\item Lattice paths (with {\em any} fixed set of jumps)  with catastrophes of length $n$ can be generated 
\begin{compactitem} 
\item in time $\LandauO(n \ln n)$ with $\LandauO(n)$ memory,
\item or in time $\LandauO(n^{3/2})$ with $\LandauO(1)$ memory  (if output is given as a stream).
\end{compactitem} 
\end{compactitem} 
\end{theorem}\vspace{-2mm}
\begin{proof}
First, via the bijection of Theorem~\ref{theo:archbijection},  
the linear-time approach of~\citet*{BacherBodiniJacquot13} for Motzkin trees can be applied to Dyck paths with catastrophes.  
The other cases can be tackled via two approaches. A first approach is to see that
classical Dyck paths (and generalized Dyck paths) can be generated by pushdown automata, or equivalently, 
by a context-free grammar. The same holds trivially for generalized lattice paths with catastrophes.
Then, using the recursive method of \citet*{Flajolet94}
(which can be seen as a wide generalization to combinatorial structures 
of what~\citet*{Hickey83} did for context-free grammars),
such paths of length $n$ can be generated in $\LandauO(n \ln n)$ average time.
\citet*{Goldwurm95} proved that this can be done with the same time complexity, with only $\LandauO(n)$ memory.
The Boltzmann method introduced  by \citet*{Flajolet04} is also 
a way to get a linear average time random generator for paths of length within $[(1-\epsilon) n, (1+\epsilon) n]$.

A second approach relies on a generating tree approach~\citet*{hexa}, where each transition is computed  via
 $$\PR{\left(\begin{cases}\text{jump $j$ when at altitude $k$, and length $m$,} \\
\text{ending at $i$ at length $n$}\end{cases}\right)}
= \frac{ f^{k+j}_{n-(m+1),i} }  {f^k_{n-m,i}},$$
where $f^k_{n,i}$ is the number of paths with catastrophes of length $n$, starting at altitude $k$ and ending at altitude $i$.
Then, for fixed $i$ and each $k$, the theory of D-finite functions applied to the algebraic functions derived similarly to Theorem~\ref{theo:LukaCatGF}
allows us to get the recurrence for the corresponding $f_n$~(see the discussion on this in~\citet*{BanderierDrmota15}). In order to get the $n$-th term $f_n$ of such recursive sequences,
there is a $\LandauO(\sqrt n)$ algorithm due to~\citet*{ChCh86}.
It is possible to win space and bit complexity
by computing the $f_n$'s in floating point arithmetic, 
instead of rational numbers (although all the $f_n$ are integers, it is often the case 
that the leading term of such recurrences is not 1, and thus 
it then implies rational number computations, and time loss in gcd computations).
All of this leads to a cost $\sum_{m=1}^n   \LandauO(\sqrt{m}) = \LandauO(n^{3/2})$, 
moreover, a $\LandauO(1)$ memory is enough to output the $n$ jumps of the lattice path, step after step, as a stream.
\end{proof}

Note that this $\LandauO(n^{3/2})$ complexity is hiding a dependency in $c+d$ in its constant.
The cost of getting each D-finite recurrence 
indeed depends on the largest upward and downward jumps $+d$ and~$-c$. 
Some computer algebra methods
for getting these recurrences (via the Platypus algorithm from~\cite{BaFl02}, 
or via integral contour representation) are analysed in \citet*{Dumont16}.
\pagebreak

\section{Conclusion}
\label{sec:catconclusion}

In this article, motivated by a natural model in queuing theory where one allows a ``reset'' of the queue,
we analysed the corresponding combinatorial model: lattice paths with catastrophes.
We showed how to enumerate them, how to get closed forms for their generating functions.
\smallskip

En passant, we gave a bijection (Theorem~\ref{theo:archbijection}) 
which extends directly to lattice paths with a $-1$-jump and an arbitrary set of positive jumps (they are sometimes called {\L}ukasiewiecz paths). 
{\L}ukasiewiecz paths with catastrophes could be considered as a kind of Galton--Watson process with catastrophes, in which some pandemic suddenly kills the full population.
Our results quantify the probability of such a pandemic over long periods.
\smallskip

It is known that the limiting objects associated to classical Dyck paths behave like Brownian excursions or Brownian meanders (see~\citet*{Marchal03}).
For our walks, Theorem~\ref{theo:limitretzero} gives some bounds on the length of the longest arch,
which, in return, proves that excursions with catastrophes (if one divides their length by the length of the longest arch) have a non-trivial continuous limiting object.
Moreover, it was interesting to see what type of behaviour these lattice paths with catastrophes exhibit.
This is illustrated by our results on the asymptotics and on the limit laws of several parameters.
We note that it is unusual to see that this leads to ``periodic'' limit laws (see Figure~\ref{fig:fractalluca}).
In fact, all these phenomena are well explained by our analytic combinatorics approach,
which also gives the speed of convergence towards these limit laws. 
\smallskip

Naturally, it could be also possible to derive some of these results
with other tools. One convenient way would be the following.
To any set of jumps, one can associate a probabilistic model with
drift~$0$; this is done by Cram\'er's trick of shifting the mean, see \citet[p.~11]{cramer1938nouveau}:
It is using  $R(u):= P(u\lambda)/P(\lambda)$ for a real $\lambda$ such that $R'(1)=0$.
A trivial computation shows that this implies that $\lambda$ is then exactly equal to $\tau$, the
unique real positive saddle point of $P(u)$. 
This often leads to a rescaled model which is analysable by the tools of Brownian motion theory. 
This trick would work for the first two cases in the trichotomy of behaviours mentioned in our theorems,
and would fail for the third case, as the process is then trivially killed by any Brownian motion renormalization.
In this last case, other approaches are needed to access to the discrete limit laws.
Analytic combinatorics seems the right tool here and
it is pleasant to rephrase some of its results in terms of some probabilistic intuition.
Indeed, our {\em analytic} quantities have thus a {\em probabilistic} interpretation:
$P''(\tau)$ could be seen as a ``variance'', $\eta$ (from the Puiseux expansion in Theorem~\ref{theo:Dasym})
could be seen as ``the multiplicative constant in the tail estimate of having an arch of length $\geq n$'' (this tail behaves like $\sim \eta/\sqrt{n}$).
However, if one uses this natural probabilistic approach, the details needed for the proofs are technical, 
and we think that analytic combinatorics is here a more suitable way to directly establish rigorous asymptotics.
\smallskip

One advantage of probability theory is to offer more flexibility in the model: the limit laws will remain the same for small perturbations of the model.
So it is natural to ask which type of flexibility analytic combinatorics can also offer.
Like we sketched in Remark~\ref{remarkQ}, our results can indeed include many variations on the model.
E.g.~if catastrophes are allowed everywhere, except at some given altitudes belonging to a set ${\mathcal A}$,
one has $Q(z)=zq \left(M (z) -\sum_{j\in {\mathcal A}} M_j(z) \right)$.
A natural combinatorial model would be for example lattice paths with catastrophes allowed only at even altitude $>0$, or at even time.
They can be analysed with the approach presented in this article.
Another interesting variant would be lattice paths where catastrophes are allowed at any altitude $h$, 
with a probability $1/h$ to have a catastrophe 
and a probability $(1-1/h)$ to have one of the jumps encoded by $P(u)$. This leads to functional equations involving a partial derivative, 
which are, however, possible to solve. Some other models are walks involving catastrophes and windfalls (a direct jump to some high altitude), 
as considered by~\citet*{Krinik08},
or walks with a direct jump to their last maximal (or minimal) altitude (see~\citet*{Schehr15}).
For all these models, further limit laws like the height, the area, the size of largest arch or the waiting time for the last catastrophe, are interesting non-trivial parameters 
which can in fact be tackled via  our approach.
Let us know if you intend to have a look on some of these models!
\smallskip

In conclusion, we have here one more application of the motto emerging from~\cite{flaj09} about problems which 
can be expressed by a combinatorial specification:\\[1mm]
\centerline{``{\em If you can specify it, you can analyse it!}'' }\\[1mm]
Indeed, it is pleasant that the tools of analytic combinatorics and the kernel method 
allowed us to solve a variant of lattice path problems,
giving their exact enumeration and the corresponding asymptotic expansions, and, additionally, 
offered efficient algorithms for uniform random generation.

\bigskip 
\bigskip 
 \textbf{Remark on this version.} This article is the long-extended version of the article with the same title 
which appeared in the volume dedicated to the GASCom'2016 conference (see~\citet*{BanderierWallner17}). In this long version,
we included more details, we gave the proofs for the asymptotic results, 
and we also added the analysis of three new parameters: 
Subsections~\ref{TotalAmplitude} (cumulative size of catastrophes), \ref{AverageAmplitude} (average size of a catastrophe), 
\ref{FirstCatastrophe} (waiting time for the first catastrophe). 
We also added the Section~\ref{sec:catuniformgeneration} dedicated to uniform random generation issues.
\bigskip 

\bigskip 

\bigskip \textbf{Acknowledgments. } The authors thank Alan Krinik and Gerardo Rubino 
who suggested considering this model of walks with catastrophes
during the Lattice Paths Conference'15, which was held in August 2015 in California in Pomona.
We also thank Jean-Marc F\'edou and the organizers of the GASCom 2016 Conference, which was held in Corsica in June 2016,
where we presented the first short version of this work.
We are also pleased to thank Gr\'egory Schehr and Rosemary J. Harris for the interesting links 
with the notion of ``resetting'' which they pinpointed to us, 
and Philippe Marchal for his friendly probabilistic feedback on this work.
Last but not least, we thank our efficient hawk-eye referee!\\
The second author was partially supported by the Austrian Science Fund (FWF) grant SFB F50-03.

\newpage

\addcontentsline{toc}{chapter}{References}
\bibliographystyle{plainnat_cyr}
\bibliography{catastrophes_journalversion}
\label{sec:biblio}

\end{document}